\documentclass[10pt,twoside,twocolumn]{IEEEtran}

\usepackage{todonotes}

\usepackage{enumitem}

\usepackage{ifthen}
\usepackage{graphicx}
\graphicspath{{./img/}}

\usepackage{booktabs}
\usepackage{cite} 
\usepackage[cmex10]{amsmath}
\usepackage{amssymb} 
\usepackage{amsfonts}
\usepackage{amsthm}

\usepackage{fixltx2e}

\usepackage{algorithm} 
\usepackage{algorithmic}

\usepackage{subfigure}

\newcommand{\eg}{{\it e.g.}}  
\newcommand{\ie}{{\it i.e.}}

\newcommand{\reals}{{\mathbb R}}

\newcommand{\naturals}{\mathbb{N}}
\newcommand{\expect}{{\mathbb E}}
\newcommand{\prob}{{\mathbb P}}

\newcommand\plusvee{\mathrel{\ooalign{\lower.5ex
\hbox{$\scriptstyle\vee\mkern.5mu$}\cr\hidewidth\raise.450ex
\hbox{$\scriptstyle+$}\cr}}}

\newcommand{\argmin}{\operatornamewithlimits{argmin}}

\newcommand{\minimize}{\operatornamewithlimits{minimize}}

\newcommand {\mtxnorm}[1]{\left|\mathopen{}\left|\mathopen{}\left|
        {#1} \right|\mathclose{}\right|\mathclose{}\right|}

\newboolean{draft}

\newcommand{\isdraft}[2]{\ifthenelse{\boolean{draft}}{#1}{#2}}

\newtheorem{theorem}{Theorem} 
\newtheorem{lemma}[theorem]{Lemma}
\newtheorem{proposition}[theorem]{Proposition}

\newtheorem{assumption}[theorem]{Assumption}

\title{Simple and fast convex relaxation method for cooperative
  localization in sensor networks using range measurements}

\author{Cl\'audia Soares*,~\IEEEmembership{Student Member,~IEEE,}
  Jo\~ao Xavier,~\IEEEmembership{Member,~IEEE,} and Jo\~ao
  Gomes,~\IEEEmembership{Member,~IEEE} \thanks{Copyright (c) 2015 IEEE. Personal use of this material is permitted. However, permission to use this material for any other purposes must be obtained from the IEEE by sending a request to pubs-permissions@ieee.org.

This work was partially
    supported by Funda\c{c}\~{a}o para a Ci\^{e}ncia e a Tecnologia (grant
    SFRH/BD/72521/2010 and projects PTDC/EMS-CRO/2042/2012,
    UID/EEA/50009/2013) and EU FP7 project MORPH (grant agreement
    no. 288704). The authors are with the Institute for Systems and
    Robotics (ISR/IST), Instituto Superior T\'ecnico, Univ
    Lisboa, 1049-001 Lisboa, Portugal (e-mail: csoares@isr.ist.utl.pt;
    jxavier@isr.ist.utl.pt; jpg@isr.ist.utl.pt).}}

\begin{document} 


\maketitle

\begin{abstract}
  We address the sensor network localization problem given noisy range
  measurements between pairs of nodes. We approach the nonconvex
  maximum-likelihood formulation via a known simple convex
  relaxation. We exploit its favorable optimization properties to the
  full to obtain an approach that: is completely distributed, has a
  simple implementation at each node, and capitalizes on an optimal
  gradient method to attain fast convergence.  We offer a parallel but
  also an asynchronous flavor, both with theoretical convergence
  guarantees and iteration complexity analysis. Experimental results
  establish leading performance. Our algorithms top the accuracy of a
  comparable state of the art method by one order of magnitude, using
  one order of magnitude fewer communications.
\end{abstract}


\begin{center} \bfseries EDICS Category: OPT-CVXR NET-CONT NET-DISP
  OPT-DOPT \end{center}
\IEEEpeerreviewmaketitle

\section{Introduction}
\label{sec:introduction}

Sensor networks are becoming ubiquitous. From environmental and
infrastructure monitoring to surveillance, and healthcare
networked extensions of the human senses in contemporary technological
societies are improving our quality of life, our productivity, and our
safety. Applications of sensor networks recurrently need to be aware
of node positions to fulfill their tasks and deliver meaningful
information. Nevertheless, locating the nodes is not trivial: these
small, low cost, low power devices are deployed in large numbers, often
with imprecise prior knowledge of their locations, and are equipped
with minimal processing capabilities. Such limitations call for
localization algorithms which are scalable, fast, and parsimonious in
their communication and computational requirements. 

\subsection{Problem statement}
\label{sec:problem-statement}
The sensor network is represented as an undirected graph~$\mathcal{G}
= (\mathcal{V},\mathcal{E})$. In the node set $\mathcal{V} =
\{1,2, \dots, n\}$ we represent the sensors with unknown positions. There
is an edge $i \sim j \in {\mathcal E}$ between sensors $i$ and $j$ if
a noisy range measurement between nodes $i$ and $j$ is available (at
both of them) and nodes $i$ and $j$ can communicate with each other.
Anchors are elements with known positions and are collected in the set
${\mathcal A} = \{ 1, \ldots, m \}$. For each sensor $i \in {\mathcal
  V}$, we let ${\mathcal A}_i \subset {\mathcal A}$ be the subset of
anchors (if any) whose distance to node $i$ is quantified by a noisy
range measurement. The set~$N_{i}$ collects the neighbors of node~$i$.

Let $\reals^p$ be the space of interest ($p=2$ for planar networks,
and $p=3$ otherwise).  We denote by $x_i \in \reals^p$ the position of
sensor $i$, and by $d_{ij}$ the noisy range measurement between
sensors $i$ and $j$, available at both $i$ and
$j$. Following~\cite{ShiHeChenJiang2010}, we assume $d_{ij} =
d_{ji}$\footnote{This entails no loss of generality: it is readily
  seen that, if $d_{ij} \neq d_{ji}$, then it suffices to replace
  $d_{ij} \leftarrow (d_{ij} + d_{ji})/2$ and $d_{ji} \leftarrow
  (d_{ij} + d_{ji})/2$ in the forthcoming optimization
  problem~\eqref{eq:snlOptimizationProblem}.}. Anchor positions are
denoted by $a_{k} \in \reals^{p}$. We let $r_{ik}$ denote the noisy
range measurement between sensor $i$ and anchor $k$, available at
sensor $i$.

The distributed network localization problem addressed in this work
consists in estimating the sensors' positions $x = \{ x_i\, : \, i \in
\mathcal{V} \}$, from the available measurements $\{ d_{ij} \, : \, i
\sim j \} \cup \{ r_{ik} \, : \, i \in {\mathcal V}, k \in {\mathcal
  A}_i \}$ and known anchor positions~$a_{k} \in
  \mathcal{A}$, through 
collaborative message passing between neighboring sensors in the
communication graph~${\mathcal G}$.

Under the assumption of zero-mean, independent and
identically-distributed, additive Gaussian measurement noise, the
maximum likelihood estimator for the sensor positions is the solution
of the optimization problem
\begin{equation}
  \label{eq:snlOptimizationProblem} 
  \minimize_{x} f(x),
\end{equation} 
where
\begin{equation*} 
  f(x) = \sum _{i \sim j} \frac{1}{2}(\|x_{i} - x_{j}\| - d_{ij})^2 + \sum_{i}
  \sum_{k \in \mathcal{A}_{i}} \frac{1}{2}(\|x_{i}-a_{k}\| - r_{ik})^2.
\end{equation*} 
Problem~\eqref{eq:snlOptimizationProblem} is nonconvex and difficult
to solve~\cite{SimonettoLeus2014},  nevertheless, it is
  guaranteed to have a global minimum, since function~$f$ is
  continuous and coercive (because, as shown in
  Lemma~\ref{lem:basic-properties} ahead, it is lower bounded by a
  coercive function~$\hat f$).




\subsection{Contributions}
\label{sec:contributions} 



We set forth 
a convex underestimator of the maximum likelihood cost for the sensor
network localization problem~\eqref{eq:snlOptimizationProblem} based
on the convex envelopes of its
parcels.

We present an optimal synchronous and parallel algorithm to minimize
this convex underestimator --- with proven convergence guarantees. We
also propose an asynchronous variant of this algorithm and prove it
converges almost surely.  Furthermore, we analyze its iteration
complexity.

Moreover, we assert the superior performance of our algorithms by
computer simulations; we compared several aspects of our method
with~\cite{SimonettoLeus2014},~\cite{OguzGomesXavierOliveira2011},
and~~\cite{GholamiTetruashviliStromCensor2013}, and our approach
always yields better performance metrics. When compared with the
method in~\cite{SimonettoLeus2014}, which operates under the same
conditions, our method outperforms it by one order of magnitude in
accuracy and in communication volume.

\subsection{Related work}
\label{sec:literature-review}



With the advent of large-scale networks, the computational paradigm of
information processing algorithms --- centralized \emph{versus}
distributed --- becomes increasingly critical. A centralized method
can be less suited for a network with meager
communication and computation resources, while a distributed algorithm
might not be adequate if the network is supposed to deliver in one
place the global result of its computations. Further, none of the
available techniques to address
Problem~\eqref{eq:snlOptimizationProblem} claims convergence to the
global optimum--- due to the nonconvexity, but also due to
ambiguities in the network topology which create more than one distant
global optimum~\cite{DestinoAbreu2011}. 

\subsubsection*{Centralized paradigm}
\label{sec:centralized-methods}
The centralized approach to the problem of sensor network localization
summoned up a wide body of research. It involves a central processing
unit to which all sensor nodes communicate their collected
measurements. Centralized architectures are prone to data traffic
bottlenecks close to the central node. Resilience to
failure, security and privacy issues are, also, not naturally
accounted for by the centralized architecture. Moreover, as the number
of nodes in the network grows, the problem to be solved at the central
node becomes increasingly complex, thus raising scalability
concerns.

Focusing on recent work, several different approaches are available,
such as the work in~\cite{KellerGur2011}, where sensor network
localization is formulated as a regression problem over adaptive
bases. The method has an initialization step using eigendecomposition
of an affinity matrix; its entries are functions of squared distance
measurements between sensors. 
The refinement is done by conjugate
gradient descent over a discrepancy function of squared distances ---
which is mathematically more tractable but amplifies measurement
errors and outliers and does not benefit from the limiting properties
of maximum likelihood estimators. This approach is closely related to
multidimensional scaling, where the sensor network localization
problem is posed as a least-squares problem, as
in~\cite{ShangRumiZhangFromherz2004}. Multidimensional scaling is
unreliable in large-scale networks due to their sparse
connectivity. Also relying on the well-tested weighted least squares
approach, the work in~\cite{DestinoAbreu2011} performs successive
minimizations of a weighted least squares cost function convolved with
a Gaussian kernel of decreasing variance.

Another successfully pursued approach is to perform semi-definite or
weaker second-order cone relaxations of the original nonconvex
problem~\eqref{eq:snlOptimizationProblem}~\cite{OguzGomesXavierOliveira2011,BiswasLiangTohYeWang2006}.
These approaches do not scale well, since the centralized SDP or SOCP
problem gets very large even for a small number of nodes.
In~\cite{OguzGomesXavierOliveira2011} and~\cite{KorkmazVeen2009} the
majorization-minimization framework was used with quadratic cost
functions to derive centralized approaches to the sensor network
localization problem.

\subsubsection*{Distributed paradigm}
\label{sec:related-work-d}

In the present work, the expression \emph{distributed method}
denotes an algorithm requiring no central or fusion node where all
nodes perform the same types of computations. Distributed approaches
for cooperative localization have been less frequent than
centralized ones, despite the more suited nature of this computational
paradigm to sensor networks, when the target application does not
require that the estimate of all sensor positions be available in one
place.

We consider two main approaches to the distributed sensor network
localization problem: 1) one where the nonconvex
Problem~\eqref{eq:snlOptimizationProblem} (or some other nonconvex
discrepancy minimization) is attacked directly, and hence the quality
of the solution is highly dependent on the quality of the algorithm's
initialization; 2) and another, where the original nonconvex sensor
network localization problem is relaxed to a convex problem, whose
tightness will determine how close the solution of the convex problem
will approximate the global solution of the original problem, not
needing any particular initialization.

\paragraph{Initialization dependent}
\label{sec:init-depend}
In reference~\cite{CostaPatwariHero2006} the authors develop a
distributed implementation of multidimensional scaling for solution
refinement. These authors base their method on the majorization-minimization
framework, but they do not provide a formal proof of convergence for
the Jacobi-like iteration.  The work in~\cite{CalafioreCarloneWei2010}
puts forward two distributed methods optimizing the discrepancy of
squared distances: a gradient algorithm with Barzilai-Borwein step
sizes calculated in a first consensus phase, followed by a gradient
computation phase, and a Gauss-Newton algorithm also with a consensus
phase and a gradient computation phase. Both are refinement methods
that need good initializations to converge to the global optimum.

\paragraph{Initialization independent}
\label{sec:init-indep}
The work in~\cite{SrirangarajanTewfikLuo2008} proposes a parallel
distributed algorithm. However, the sensor network localization
problem adopts the previously discussed squared distances discrepancy
function. Also, each sensor must solve a second order cone program at
each algorithm iteration, which can be a demanding task for the simple
hardware used in sensor networks' motes. Furthermore, the formal
convergence properties of the algorithm are not established.
The work in~\cite{ChanSo2009} also considers network localization
outside a maximum likelihood framework. The approach proposed
in~\cite{ChanSo2009} is not parallel, operating sequentially through
layers of nodes: neighbors of anchors estimate their positions and
become anchors themselves, making it possible in turn for their
neighbors to estimate their positions, and so on. Position estimation
is based on planar geometry-based heuristics.
In~\cite{KhanKarMoura2010}, the authors propose an algorithm with
assured asymptotic convergence, but the solution is computationally
complex since a triangulation set must be calculated, and matrix
operations are pervasive. Furthermore, in order to attain good
accuracy, a large number of range measurement rounds must be acquired,
one per iteration of the algorithm, thus increasing energy
expenditure.  On the other hand, the algorithm presented
in~\cite{ShiHeChenJiang2010} and based on the non-linear Gauss Seidel
framework, has a pleasingly simple implementation, combined with the
convergence guarantees inherited from the framework. Notwithstanding,
this algorithm is sequential, \ie, nodes perform their calculations in
turn, not in a parallel fashion. This entails the existence of a
network-wide coordination procedure to precompute the processing
schedule upon startup, or whenever a node joins or leaves the network.
The sequential nature of the work in~\cite{ShiHeChenJiang2010} was
superseded by the work in~\cite{SimonettoLeus2014} which puts forward
a parallel method based on two consecutive relaxations of the maximum
likelihood estimator in~\eqref{eq:snlOptimizationProblem}. The first
relaxation is a semi-definite program with a rank relaxation, while
the second is an edge based relaxation, best suited for the
Alternating Direction Method of Multipliers (ADMM). 
The main drawback is the amount of communications required to manage
the ADMM variable local copies, and by the prohibitive complexity of
the problem at each node. In fact, each one of the simple sensing
units must solve a semidefinite program at each ADMM iteration and
after the update copies of the edge variables must be exchanged with
each neighbor.  A simpler approach was devised
in~\cite{GholamiTetruashviliStromCensor2013} by extending the source
localization Projection Onto Convex Sets algorithm
in~\cite{BlattHero2006} to the problem of sensor network
localization. The proposed method is sequential, activating nodes one
at a time according to a predefined cyclic schedule; thus, it does not
take advantage of the parallel nature of the network and imposes a
stringent timetable for individual node activity.

\section{Convex relaxation}
\label{sec:convex-relaxation}
Problem~\eqref{eq:snlOptimizationProblem} can be written as
\begin{equation}
  \label{eq:non-cvx-prob-dist}
  \minimize_{x} \sum_{i \sim
    j}\frac{1}{2}\mathrm{d}_{\mathrm{S}_{ij}}^{2}(x_{i}-x_{j}) + \sum_{i}\sum_{k
    \in \mathcal{A}_{i}} \frac{1}{2}\mathrm{d}^{2}_{\mathrm{S_{a}}_{ik}}(x_{i}),
\end{equation}
where $\mathrm{d}^{2}_{C}(x)$ represents the squared Euclidean distance
of point $x$ to the set $C$, \ie,
$
  \mathrm{d}^{2}_{C}(x) = \inf_{y \in C} \|x -y\|^{2},
$
and the sets $\mathrm{S}_{ij}$ and $\mathrm{S_{a}}_{ik}$ are defined
as the spheres generated by the noisy measurements $d_{ij}$ and $r_{ik}$
$$
  \mathrm{S}_{ij} = \left\{z : \|z\| =
    d_{ij}\right\}, \quad
  \mathrm{S_{a}}_{ik}  =  \left\{z  : \|z - a_{k}\| =
    r_{ik}\right\}.
$$
nonconvexity of~\eqref{eq:non-cvx-prob-dist} follows from the nonconvexity of the building block 
\begin{equation}
  \label{eq:sphere-sq-dist}
  \frac{1}{2}\mathrm{d}_{\mathrm{S}_{ij}}^{2}(z) =
  \frac{1}{2}\inf_{\|y\| = d_{ij}} \|z-y\|^{2}.
\end{equation} 
A simple convexification consists in replacing it by 
\begin{equation}
  \label{eq:ball-sq-dist}
    \frac{1}{2}\mathrm{d}_{\mathrm{B}_{ij}}^{2}(z) = \frac{1}{2}
    \inf_{\|y\| \leq d_{ij}} \|z-y\|^{2}
\end{equation}
where $ \mathrm{B}_{ij} = \left\{z \in \reals^{p} : \|z\| \leq
  d_{ij}\right\}, $ is the convex hull of $\mathrm{S}_{ij}$. Actually,
\eqref{eq:ball-sq-dist} is the convex envelope\footnote{The convex
  envelope (or convex hull) of a function~$\gamma$ is its best
  possible convex underestimator, \ie, $ \text{conv } \gamma(x) =
  \sup\left \{ \eta(x) \; : \; \eta \leq \gamma, \; \eta \text{ is
      convex} \right \} $, and is hard to determine in general.}
of~\eqref{eq:sphere-sq-dist}.  This fact is illustrated in
Figure~\ref{fig:tightenvelope} with a one-dimensional example; a formal proof for the generic case is given in section~\ref{sec-cvxenv}.
\begin{figure}[h]
  \centering
  \includegraphics[width=\columnwidth]{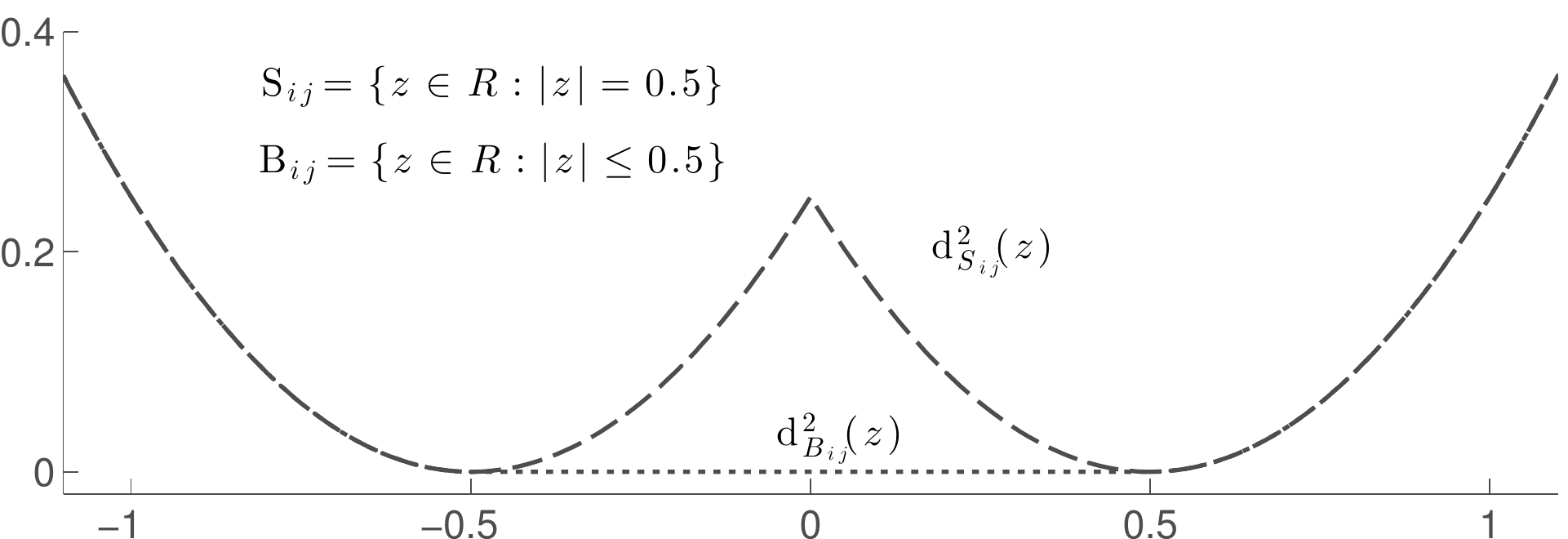}
  \caption{Illustration of the convex envelope for intersensor terms of the
    nonconvex cost function~\eqref{eq:non-cvx-prob-dist}. The squared
    distance to the ball $\mathrm{B}_{ij}$ (dotted line) is the convex
    hull of the squared distance to the sphere $\mathrm{S}_{ij}$
    (dashed line). In this one dimensional example the value of the range
    measurement is $d_{ij} = 0.5$ }
  \label{fig:tightenvelope}
\end{figure}
The terms of~\eqref{eq:non-cvx-prob-dist} associated
with anchor measurements are similarly relaxed as
\begin{equation}
  \label{eq:ball-sq-dist-anchors}
  \mathrm{d}_{\mathrm{B_{a}}_{ik}}^{2}(z) = \inf_{\|y - a_{k}\| \leq
    r_{ik}} \|z-y\|^{2},
\end{equation}
where the set $\mathrm{B_{a}}_{ik}$ is the convex hull of
$\mathrm{S_{a}}_{ik}$:
$
    \mathrm{B_{a}}_{ik} = \left\{z \in \reals^{p} : \|z - a_{k}\| \leq
    r_{ik}\right\}.
$
Replacing the nonconvex parcels in~\eqref{eq:non-cvx-prob-dist} by the
sums of terms~\eqref{eq:ball-sq-dist}
and~\eqref{eq:ball-sq-dist-anchors} we obtain the convex problem
\begin{equation}
  \label{eq:cvx-prob-1}
    \minimize_{x} \hat{f}(x) = \sum_{i \sim
    j}\frac{1}{2}\mathrm{d}_{\mathrm{B}_{ij}}^{2}(x_{i}-x_{j}) + \sum_{i}\sum_{k
    \in \mathcal{A}_{i}} \frac{1}{2}\mathrm{d}^{2}_{\mathrm{B_{a}}_{ik}}(x_{i}).
\end{equation}

The function in Problem~\eqref{eq:cvx-prob-1} is an underestimator
of~\eqref{eq:non-cvx-prob-dist} but it is not the convex envelope of
the original function.  We argue that in our application of sensor
network localization it is generally a very good approximation whose
sub-optimality can be quantified, as discussed in
Section~\ref{sec:qual-conv-probl}.  The cost
function~\eqref{eq:cvx-prob-1} also appears
in~\cite{GholamiTetruashviliStromCensor2013} albeit via a distinct
reasoning; our convexification mechanism seems more intuitive. But the
striking difference with respect
to~\cite{GholamiTetruashviliStromCensor2013} is
how~\eqref{eq:cvx-prob-1} is exploited to generate distributed
solution methods. Whereas ~\cite{GholamiTetruashviliStromCensor2013}
lays out a sequential block-coordinate approach, we show
that~\eqref{eq:cvx-prob-1} is amenable to distributed solutions either
via the fast Nesterov's gradient method (for synchronous
implementations) or exact/inexact randomized block-coordinate methods
(for asynchronous implementations).

\section{Distributed sensor network localization}
\label{sec:distr-sens-netw}

 We propose two distributed algorithms: a synchronous
  one, where nodes work in parallel, and an asynchronous, gossip-like
  algorithm, where each node starts its processing step according to
  some probability distribution. Both algorithms require to compute
  the gradient of the cost function and its Lipschitz constant. In
  order to achieve this it is convenient to rewrite
Problem~\eqref{eq:cvx-prob-1} as
\begin{equation}
  \label{eq:cvx-prob-2}
    \minimize_{x} \frac{1}{2}\mathrm{d}_{\mathrm{B}}^{2}(Ax) + \sum_{i}\sum_{k
    \in \mathcal{A}_{i}} \frac{1}{2}\mathrm{d}^{2}_{\mathrm{B_{a}}_{ik}}(x_{i}), 
\end{equation}
where $A=C \otimes I_{p}$, $C$ is the arc-node incidence matrix of
$\mathcal{G}$,~$I_{p}$ is the identity matrix of size~$p$, and
$\mathrm{B}$ is the Cartesian product of the balls~$\mathrm{B}_{ij}$
corresponding to all the edges in~$\mathcal{E}$. 
We
denote the two parcels in~\eqref{eq:cvx-prob-2} as
\begin{equation*}
  \label{eq:g-and-h}
  g(x) = \frac{1}{2}\mathrm{d_{B}}^{2}(Ax), \qquad h(x) = \sum_{i} h_{i}(x_{i}),
\end{equation*}
where $h_{i}(x_{i}) = \sum_{k \in \mathcal{A}_{i}} \frac{1}{2}
\mathrm{d}^{2}_{\mathrm{B_{a}}_{ik}}(x_{i})$. 
Problems~\eqref{eq:cvx-prob-1} and~\eqref{eq:cvx-prob-2} are equivalent since $A x$ is the vector  $\left( x_i - x_j\,:\,i \sim j \right)$ and
function~$g(x)$ in~\eqref{eq:cvx-prob-2} can be written as
\begin{align*}
  g(x) & = \frac{1}{2}\mathrm{d_{B}}^{2}(Ax) \\
&= \frac 12 \inf_{y \in \mathrm{B}}\|Ax - y\|^{2} \\
  &= \frac 12 \inf_{\|y_{ij}\| \leq d_{ij}}\sum_{i \sim j} \|x_{i}
  - x_{j} - y_{ij}\|^{2}
\end{align*}
and as all the terms are non-negative and the constraint set is a
Cartesian product, we can exchange~$\inf$ with the summation,
resulting in
\begin{align*}
  g(x) &= \frac 12 \sum_{i \sim j} \inf_{\|y_{ij}\| \leq d_{ij}}
  \|x_{i} - x_{j} - y_{ij}\|^{2}\\
  &= \sum_{i \sim j}  \frac 12 \mathrm{d_{B}}_{ij}^{2}(x_{i}-x_{j}),
\end{align*}
which is the corresponding term in~\eqref{eq:cvx-prob-1}.

\subsection{Gradient and Lipschitz constant of~$\hat{f}$}
\label{sec:grad-lipsch-const}

To simplify notation, let us define the functions:
\begin{equation*}
  \phi_{\mathrm{B}_{ij}}(z) = \frac 12
  \mathrm{d}^{2}_{\mathrm{B}_{ij}}(z), \qquad 
  \phi_{\mathrm{B_{a}}_{ik}}(z) = \frac 12
  \mathrm{d}^{2}_{\mathrm{B_{a}}_{ik}}(z).
\end{equation*}
Now we call on a key result from convex analysis (see~\cite[Prop. X.3.2.2,
Th. X.3.2.3]{UrrutyMarechal1993}): the function
in~\eqref{eq:ball-sq-dist},~$\phi_{\mathrm{B}_{ij}}(z) = \frac12
\mathrm{d}_{\mathrm{B}_{ij}}^{2}(z)$
is convex, differentiable, and its gradient is
\begin{equation}
  \label{eq:sq-dist-grad}
  \nabla \phi_{\mathrm{B}_{ij}}(z) = z - \mathrm{P}_{\mathrm{B}_{ij}}(z),
\end{equation}
where $\mathrm{P}_{\mathrm{B}_{ij}}(z)$ is the orthogonal projection
of point $z$ onto the closed convex set $\mathrm{B}_{ij}$
\begin{equation*}
  \mathrm{P}_{\mathrm{B}_{ij}}(z) = \argmin_{y \in \mathrm{B}_{ij}} \|z-y\|.
\end{equation*}
Further,  function $\phi_{\mathrm{B}_{ij}}$ has a Lipschitz
continuous gradient with constant $L_{\phi} = 1$, \ie,
\begin{equation}
  \label{eq:sq-dist-lipschitz-grad}
  \|\nabla\phi_{\mathrm{B}_{ij}}(x)-\nabla\phi_{\mathrm{B}_{ij}}(y)\|
  \leq \|x-y\|.
\end{equation}
We show~\eqref{eq:sq-dist-lipschitz-grad} in section~\ref{sec-lips}.

Let us define a vector-valued function~$\phi_{\mathrm{B}}$, obtained
by stacking all functions~$\phi_{\mathrm{B}_{ij}}$. Then,~$g(x) =
\phi_{\mathrm{B}}(Ax)$. From this relation, and using
Eq.~\eqref{eq:sq-dist-grad}, we can compute the gradient
of~$g(x)$:
\begin{eqnarray}
\nonumber
  \nabla g(x) &=& A^{\top}\nabla\phi_{\mathrm{B}}(Ax)\\
    \nonumber
    &=& A^{\top} (Ax - \mathrm{P}_{\mathrm{B}}(Ax))\\
    \label{eq:gradg}
    &=& \mathcal{L}x - A^{\top}\mathrm{P}_{\mathrm{B}}(Ax),
\end{eqnarray}
where the second equality follows from~\eqref{eq:sq-dist-grad} and
$\mathcal{L}= A^{\top}A = L \otimes I_{p}$, with $L$ being the
Laplacian matrix of~$\mathcal{G}$. This gradient is Lipschitz
continuous and we can obtain an easily computable Lipschitz
constant~$L_{g}$ as follows
\begin{eqnarray}
  \nonumber
  \|\nabla g(x) - \nabla g(y) \| &=&
  {\|A^{\top} \left(\nabla\phi_{\mathrm{B}}(Ax) -
    \nabla\phi_{\mathrm{B}}(Ay)\right)\|}\\
  \nonumber
  &\leq& \mtxnorm{A} \|Ax - Ay\|\\
  \nonumber
  &\leq& \mtxnorm{A}^{2} \|x-y\|\\
  \nonumber
  &=& \lambda_{\mathrm{max}}(A^{\top}A)\|x-y\|\\
  \nonumber
  &\overset{\tiny{(a)}}{=}& \lambda_{\mathrm{max}}(L)\|x-y\|\\
  \label{eq:lips-g}
  &\leq& \underbrace{2 \delta_{\mathrm{max}}}_{L_{g}} \|x-y\|,
\end{eqnarray}
where~$\mtxnorm{A}$ is the maximum singular value norm; equality $(a)$
is a consequence of Kronecker product properties. In~\eqref{eq:lips-g}
we denote the maximum node degree of $\mathcal{G}$ by
$\delta_{\mathrm{max}}$. A proof of the
bound~$\lambda_{\mathrm{max}}(L) \leq 2\delta_{\mathrm{max}}$ can be
found in~\cite{chung1997spectral}\footnote{{A tighter bound would be
  $\lambda_{\mathrm{max}}(L) \leq \max_{i \sim j} \left \{\delta_{i} +
    \delta_{j} - c(i,j) \right \}$ where~$\delta_{i}$ is
the degree of node~$i$ and~$c(i,j)$ is the number of vertices that are
adjacent to both~$i$ and~$j$~\cite[Th.~4.13]{Bapat2010}, nevertheless~$2
  \delta_{\mathrm{max}}$ is easier to compute in a distributed
  way.}}.

The gradient of $h$ is $ \nabla h(x) = \left( \nabla h_{1}(x_{1}),
  \ldots, \nabla h_{n}(x_{n}) \right), $ where the gradient of each
$h_{i}$ is
\begin{equation}
  \label{eq:gradhi}
  \nabla h_{i}(x_{i}) = \sum_{k \in \mathcal{A}_{i}} \nabla
  \phi_{\mathrm{B_{a}}_{ik}}(x_{i}).
\end{equation}
The gradient of $h$ is also Lipschitz
continuous. The constants~$L_{h_{i}}$ for $\nabla h_{i}$ are
\begin{eqnarray}
  \nonumber
  \| \nabla h_{i}(x_{i}) - \nabla h_{i}(y_{i}) \| & \leq & \sum_{k \in
  \mathcal{A}_{i}} \| \nabla \phi_{\mathrm{B_{a}}_{ik}}(x_{i}) -
\nabla \phi_{\mathrm{B_{a}}_{ik}}(y_{i}) \|\\   \label{eq:lips-hi}
& \leq & |\mathcal{A}_{i}| \|x_{i}-y_{i}\|,
\end{eqnarray}
where~$|\mathcal{C}|$ is the cardinality of set~$\mathcal{C}$. We now
have an overall constant~$L_{h}$ for $\nabla h$,
\begin{eqnarray}
  \nonumber
  \|\nabla h(x) - \nabla h(y) \| & = & \sqrt{\sum_{i}\| \nabla
    h_{i}(x_{i}) - \nabla h_{i}(y_{i}) \|^{2}}\\
  \nonumber
  & \leq & \sqrt{\sum_{i} |\mathcal{A}_{i}|^2 \|x_{i}-y_{i}\|^{2}}\\
  \label{eq:lips-h}
  &\leq& \underbrace{\max(|\mathcal{A}_{i}| : i \in
    \mathcal{V})}_{L_{h}} \|x-y\|.
\end{eqnarray}
{We are now able to write $\nabla \hat f$, the gradient of our cost
function, as
\begin{equation}
  \label{eq:grad-f-hat}
  \nabla \hat f(x) = \mathcal{L}x -A^{\top}\mathrm{P_{B}}(Ax) +
  \begin{bmatrix}
    \sum_{k \in \mathcal{A}_{1}}x_{1} - \mathrm{P_{Ba}}_{1k}(x_{1})\\
    \vdots\\
    \sum_{k \in \mathcal{A}_{n}}x_{n} - \mathrm{P_{Ba}}_{nk}(x_{n})
   \end{bmatrix}.
\end{equation}}
A Lipschitz constant $L_{\hat{f}}$ is, thus,
\begin{equation}
  \label{eq:lips-f_hat}
  L_{\hat{f}} = 2\delta_{\mathrm{max}} + \max(|\mathcal{A}_{i}| : i \in \mathcal{V}).
\end{equation}
This constant is easy to precompute by,
\eg, a diffusion algorithm --- c.f.~\cite[Ch. 9]{MesbahiEgerstedt2010}
for more information.

In summary, we can compute the gradient of $\hat{f}$ using
Equation~\eqref{eq:grad-f-hat} and a Lipschitz constant
by~\eqref{eq:lips-f_hat}, which leads us to the algorithms described
in Sections~\ref{sec:synchronous-method}
and~\ref{sec:asynchronous-method} for minimizing $\hat{f}$.

\subsection{Parallel method}
\label{sec:synchronous-method}

Since $\hat f$ has a Lipschitz continuous gradient we can follow
Nesterov's optimal method~\cite{Nesterov1983}. Our
approach is detailed in Algorithm~\ref{alg:synDR}.
\begin{algorithm}[tb]
  \caption{Parallel method}
  \label{alg:synDR}
  \begin{algorithmic}[1] 
    \REQUIRE $L_{\hat{f}}; \{d_{ij} : i \sim j \in \mathcal{E}\};
    \{r_{ik} : i \in \mathcal{V}, k \in \mathcal{A}\};$
    \ENSURE $\hat x$ 
    \STATE $k = 0;$
    \STATE each node $i$ chooses random $x_{i}(0) = x_{i}(-1)$;
    \WHILE{some stopping criterion is not met, each node $i$}
    \STATE $k = k+1$
    \STATE $
    \begin{aligned}[t]
      w_{i} =
      x_{i}(k-1)+\frac{k-2}{k+1}\left(x_{i}(k-1)-x_{i}(k-2)\right);
    \end{aligned}
    $ \label{alg:extrapolated-point}
    \STATE node $i$ broadcasts $w_{i}$ to its neighbors 
    \STATE $
    \begin{aligned}[t]
      \nabla g_{i}(w_{i}) = \delta_{i}w_{i} - \sum_{j \in N_{i}}w_{j}
      + \\
      + \sum_{j \in N_{i}} c_{(i \sim j,i)} \mathrm{P_{B}}_{ij}(w_{i} - w_{j});
    \end{aligned}
    $ \label{alg:gradg}
    \STATE $
    \begin{aligned}[t]
      \nabla h_{i}(w_{i}) = \sum_{k \in \mathcal{A}_{i}} w_{i} -
      \mathrm{P_{B_{a}}}_{ik}(w_{i});
    \end{aligned}
    $\label{alg:gradh}
    \STATE $
    \begin{aligned}[t]
      x_{i}(k) = w_{i} - \frac{1}{L_{\hat f}}(\nabla g_{i}(w_{i}) +
      \nabla h_{i}(w_{i}));
    \end{aligned}
    $ \label{alg:updatex}
    \ENDWHILE
    \RETURN $\hat{x} = x(k)$
  \end{algorithmic}
\end{algorithm} { Step~\ref{alg:extrapolated-point}
  computes the extrapolated points~$w_i$ in a standard application of
  Nesterov's method~\cite{Vandenberghe2014FastProxGrad}.
  Steps~\ref{alg:gradg} and~\ref{alg:gradh}, which constitute the core
  of the algorithm, correspond the $i$-th entry of~$\nabla \hat f$
  given in~\eqref{eq:grad-f-hat}. Specifically, Step~\ref{alg:gradg}
  coincides with the~$i$th entry
  of~$\mathcal{L}x-A^{\top}\mathrm{P_{B}}(Ax)$ in~\eqref{eq:gradg}
  where $c_{(i \sim j,i)}$ denotes the entry $(i \sim j,i)$ in the
  arc-node incidence matrix $C$, and~$\delta_{i}$ is the degree of
  node~$i$. The $i$-th entry of~$\mathcal{L}x$ can be computed by
  node~$i$, from its current position estimate and the position
  estimates of the neighbors, in particular, it
  holds~$(\mathcal{L}x)_{i} = \delta_{i}x_{i} - \sum_{j \in
    N_{i}}x_{j}$.
  The less obvious parallel term is $A^{\top}\mathrm{P_{B}}(Ax)$.  We
  start the analysis by the concatenated
  projections~$\mathrm{P_{B}}(Ax) =
  \{\mathrm{P_{B}}_{ij}(x_{i}-x_{j})\}_{i \sim j \in \mathcal{E}}$.
  Each one of those projections only depends on the edge terminals and
  the noisy measurement~$d_{ij}$. The product with $A^{\top}$ will
  collect, at the entries corresponding to each node, the sum of the
  projections relative to edges where it intervenes, with a positive
  or negative sign depending on the arbitrary edge direction agreed
  upon at the onset of the algorithm. More
  specifically,~$(A^{\top}\mathrm{P_{B}}(Ax))_{i} = \sum_{j \in
    N_{i}}c_{(i \sim j,i)} \mathrm{P_{B}}_{ij}(x_{i}-x_{j}),$
  as presented in Step~\ref{alg:gradg} of
  Algorithm~\ref{alg:synDR}. The last summand in~\eqref{eq:grad-f-hat}
  is simply~$\nabla h(x)$, and the $i$-th entry of~$\nabla h(x)$ is
  given in~\eqref{eq:gradhi}. This can be easily computed
  independently by each node according to Step~\ref{alg:gradh}. The
  position updates in Step~\ref{alg:updatex} of the algorithm require
  the computation of the gradient of the cost w.r.t. the coordinates
  of node~$i$, done in the previous steps, evaluated at the
  extrapolated points~$w_{i}$.}

\subsection{Asynchronous method}
\label{sec:asynchronous-method}
The method described in Algorithm~\ref{alg:synDR} is fully parallel
but still depends on some synchronization between all the nodes --- so
that their updates of the gradient are consistent.  This requirement
can be inconvenient in some applications of sensor networks; to circumvent
it, we present a fully asynchronous method, achieved by means of a
broadcast gossip scheme (c.f.~\cite{shah2009} for an extended survey
of gossip algorithms).

Nodes are equipped with independent clocks ticking at random times (say, as Poisson point processes).
When node~$i$'s clock ticks, it performs the
update of its variable~$x_{i}$ and broadcasts the update to its
neighbors. Let the order of node activation be collected
in~$\{\xi_{k}\}_{k \in \naturals}$, a sequence of independent random
variables taking values on the set~$\mathcal{V}$, such that
\begin{equation}
  \label{eq:rv}
  \prob(\xi_{k} = i) = P_{i} > 0.
\end{equation}
Then, the asynchronous update of variable~$x_{i}$ on node~$i$ can
be described as in Algorithm~\ref{alg:asynDR}.
\begin{algorithm}[tb]
  \caption{Asynchronous method}
  \label{alg:asynDR}
  \begin{algorithmic}[1] 
    \REQUIRE $L_{\hat{f}}; \{d_{ij} : i \sim j \in \mathcal{E}\};
    \{r_{ik} : i \in \mathcal{V}, k \in \mathcal{A}\};$
    \ENSURE $\hat x$
    \STATE each node $i$ chooses random $x_{i}(0)$;
    \STATE $k = 0$;
    \WHILE{some stopping criterion is not met, each node $i$}
    \STATE $k = k + 1;$
    \IF {$\xi_k = i$}
    \STATE $
    x_i(k) = \argmin_{w_{i}} \hat f(x_{1}(k-1), \dots, w_{i},
        \dots, x_{n}(k-1))$ \label{alg:asynmin}
    \ELSE
    \STATE $x_i(k) = x_i(k-1)$
    \ENDIF
    \ENDWHILE
    \RETURN $\hat x = x(k)$
  \end{algorithmic}
\end{algorithm}

To compute the minimizer in Step~\ref{alg:asynmin} of
Algorithm~\ref{alg:asynDR} it is useful to recast 
Problem~\eqref{eq:cvx-prob-2} as
\begin{equation}
  \label{eq:cvx-prob-3}
  \minimize_{x} \sum_{i} \left(\sum_{j \in
    N_{i}}\frac{1}{4}\mathrm{d}^{2}_{\mathrm{B}_{ij}}(x_{i}-x_{j}) +
  \sum_{k \in \mathcal{A}_{i}}
  \frac{1}{2}\mathrm{d}^{2}_{\mathrm{B_{a}}_{ik}}(x_{i}) \right),
\end{equation}
where the factor~$\frac{1}{4}$ accounts for the duplicate terms when
considering summations over nodes instead of over edges. By fixing the
neighbor positions, each node solves a single source localization problem;
this setup leads to the Problem
\begin{equation}
  \label{eq:sl-problem}
  \minimize_{x_{i}} \hat f_{sl_{i}}(x_{i}) := \sum_{j \in
    N_{i}}\frac{1}{4}\mathrm{d}^{2}_{\mathrm{B_{s}}_{ij}}(x_{i}) + \sum_{k \in
    \mathcal{A}_{i}} \frac{1}{2}\mathrm{d}^{2}_{\mathrm{B_{a}}_{ik}}(x_{i}),
\end{equation}
where~$\mathrm{B_{s}}_{ij} = \left\{z \in \reals^{p} : \|z - x_{j}\|
  \leq d_{ij}\right\}$.
\begin{algorithm}[tb]
  \caption{Asynchronous update at each node $i$}
  \label{alg:sl}
  \begin{algorithmic}[1]
      \REQUIRE $\xi_{k}; L_{\hat{f}}; \{d_{ij} : j \in N_{i}\}; \{r_{ik} : k
      \in \mathcal{A}_{i}\};$
      \ENSURE $x_{i}(k)$
      \IF{$\xi_{k}$ \NOT $i$} \STATE{$x_i(k) = x_{i}(k-1)$;} 
      \RETURN $x_i(k)$;
      \ENDIF
      \STATE choose random $z(0)= z(-1)$;
      \STATE $l = 0;$
      \WHILE{some stopping criterion is not met}
      \STATE $l = l + 1;$
      \STATE $
      \begin{aligned}[t]
        w = z(l-1) + \frac{l-2}{l+1}(z(l-1) - z(l-2));
      \end{aligned}
      $
      \STATE $
      \begin{aligned}[t]
        \nabla \hat f_{sl_{i}}(w) = \frac{1}{2}\sum_{j \in N_{i}} w -
      \mathrm{P_{B_{S}}}_{ij}(w) + \sum_{k \in \mathcal{A}_{i}}
        w - \mathrm{P_{B_{a}}}_{ik}(w)
      \end{aligned}
      $
      \STATE $
      \begin{aligned}[t]
        z(l) = w - \frac{1}{L_{\hat f}} \nabla \hat f_{sl_{i}}(w)
      \end{aligned}
      $
      \ENDWHILE
      \RETURN $x_{i}(k) = z(l)$
  \end{algorithmic}
\end{algorithm} {Note that the function
  in~\eqref{eq:sl-problem} is continuous and coercive; thus, the
  optimization problem~\eqref{eq:sl-problem} has a solution.}

We solve Problem~\eqref{eq:sl-problem} at each node by employing
Nesterov's optimal accelerated gradient method as described in
Algorithm~\ref{alg:sl}.  The asynchronous method proposed in
Algorithm~\ref{alg:asynDR} converges to the set of minimizers of
function~$\hat f$, as established in Theorem~\ref{th:asy-convergence},
in Section~\ref{sec:convergence-analysis}.

We also propose an inexact version in which nodes do not solve Problem~\eqref{eq:sl-problem} but instead take just one gradient step. That is,
%
%
simply
replace Step~\ref{alg:asynmin} in Algorithm~\ref{alg:asynDR} by
\begin{equation}
  \label{eq:grad-step-update}
  x_{i}(k) =
    x_{i}(k-1) - \frac 1{L_{\hat f}} \nabla_{i}\hat f (x(k-1))
   \end{equation}
where~$\nabla_{i} \hat f (x_1, \ldots, x_n)$ is the gradient with respect to $x_i$, and assume
\begin{equation}
  \label{eq:prob-uniform}
  \prob \left( \xi_{k} = i \right) = \frac 1n.
\end{equation}
The convergence terms of the resulting algorithm are established in
Theorem~\ref{th:asy-convergence-GD},
Section~\ref{sec:convergence-analysis}.

\section{Theoretical analysis}
\label{sec:convergence-analysis}

A relevant question regarding Algorithms~\ref{alg:synDR}
and~\ref{alg:asynDR} is whether they will return a good solution to
the problem they are designed to solve, after a reasonable amount of
computations. Sections~\ref{sec:conv-synchronous-method}
and~\ref{sec:conv-asynchronous-method} address convergence issues of
the proposed methods, and discuss some of the assumptions on the
problem data. Section~\ref{sec:qual-conv-probl} provides a formal
bound for the gap between the original and the convexified problems.

\subsection{Quality of the convexified problem}
\label{sec:qual-conv-probl}

While evaluating any approximation method it is important to know how
far the approximate optimum is from the original one. In this Section
we will focus on this analysis.

It was already noted in Section~\ref{sec:convex-relaxation}
that~$\phi_{\mathrm{B}_{ij}}(z) = \phi_{\mathrm{S}_{ij}}(z)$
for~$\|z\| \geq d_{ij}$; when the functions differ, for~$\|z\| <
d_{ij}$, we have that~$\phi_{\mathrm{B}_{ij}}(z) = 0$. The same
applies to the terms related to anchor measurements. The optimal value
of function~$f$, denoted by~$f^{\star}$, is bounded by
$
  \hat f^{\star} = \hat f(x^{\star}) \leq f^{\star} \leq f(x^{\star}),
$
where~$x^{\star}$ is the minimizer of the convexified
problem~\eqref{eq:cvx-prob-1}, and~$\hat f^{\star} = \inf_{x} \hat
f(x)$ is the minimum of function~$\hat f$. With these inequalities we
can compute a bound for the optimality gap,
after~\eqref{eq:cvx-prob-1} is solved, as
\begin{IEEEeqnarray}{rCl} \nonumber
  f^{\star} - \hat f^{\star} &\leq& f(x^{\star}) - \hat{f}^{\star}\\\nonumber
  & = & \sum_{i \sim j \in \mathcal{E}}
  \frac{1}{2}\left(\mathrm{d}_{\mathrm{S}_{ij}}^2(x_i^\star -
    x_j^\star) -  \mathrm{d}^2_{\mathrm{B}_{ij}}(x_i^\star -
    x_j^\star) \right)
\\\nonumber
&& + \sum_{i \in \mathcal{V}} \sum_{k \in  \mathcal{A}_i} \frac{1}{2} \left ( 
\mathrm{d}_{\mathrm{S_a}_{ik}}^2(x_i^\star) -  
\mathrm{d}^2_{\mathrm{B_a}_{ik}}(x_i^\star) \right)\\ \nonumber
  & = & \sum_{i \sim j \in \mathcal{E}_2} \frac{1}{2}
  \mathrm{d}_{\mathrm{S}_{ij}}^2(x_i^\star - x_j^\star) + \sum_{i \in
    \mathcal{V}} \sum_{k \in  \mathcal{A}_{2_i}} \frac{1}{2}
\mathrm{d}_{\mathrm{S_a}_{ik}}^2(x_i^\star).\\\label{eq:bound-tight}
\end{IEEEeqnarray}
In Equation~\eqref{eq:bound-tight}, we denote the set of edges where
the distance of the estimated positions is less than the distance
measurement by~$\mathcal{E}_{2} = \{i \sim j \in \mathcal{E} :
\mathrm{d}^2_{\mathrm{B}_{ij}}(x_i^\star - x_j^\star) = 0\}$, and
similarly~$\mathcal{A}_{2_{i}} = \{k \in \mathcal{A}_{i} :
\mathrm{d}_{\mathrm{B_a}_{ik}}^2(x_i^\star) =
0\}$. Inequality~\eqref{eq:bound-tight} suggests a simple method to
compute a bound for the optimality gap of the solution returned by the
algorithms:
\begin{enumerate}
\item Compute the optimal solution~$x^{\star}$ using
  Algorithm~\ref{alg:synDR} or~\ref{alg:asynDR};
\item Select the terms of the convexified
  problem~\eqref{eq:cvx-prob-1} which are zero;
\item Add the nonconvex costs of each of these edges, as
  in~\eqref{eq:bound-tight}.
\end{enumerate}
Our bound is tighter than the one (available \textit{a priori})  from
applying~\cite[Th. 1]{udellBoyd2014}, which is
\begin{equation}
  \label{eq:udellBoyd}
  f^{\star} - \hat f^{\star} \leq 
  \sum_{i \sim j \in \mathcal{E}} \frac{1}{2} d_{ij}^2 + 
  \sum_{i \in \mathcal{V}} \sum_{k \in  \mathcal{A}_{i}} \frac{1}{2} r_{ik}^2.
\end{equation}

For the one-dimensional example of the star
network costs depicted in Figure~\ref{fig:nonconvexities} 
\begin{figure}[h]
  \centering
  \includegraphics[width=\columnwidth]{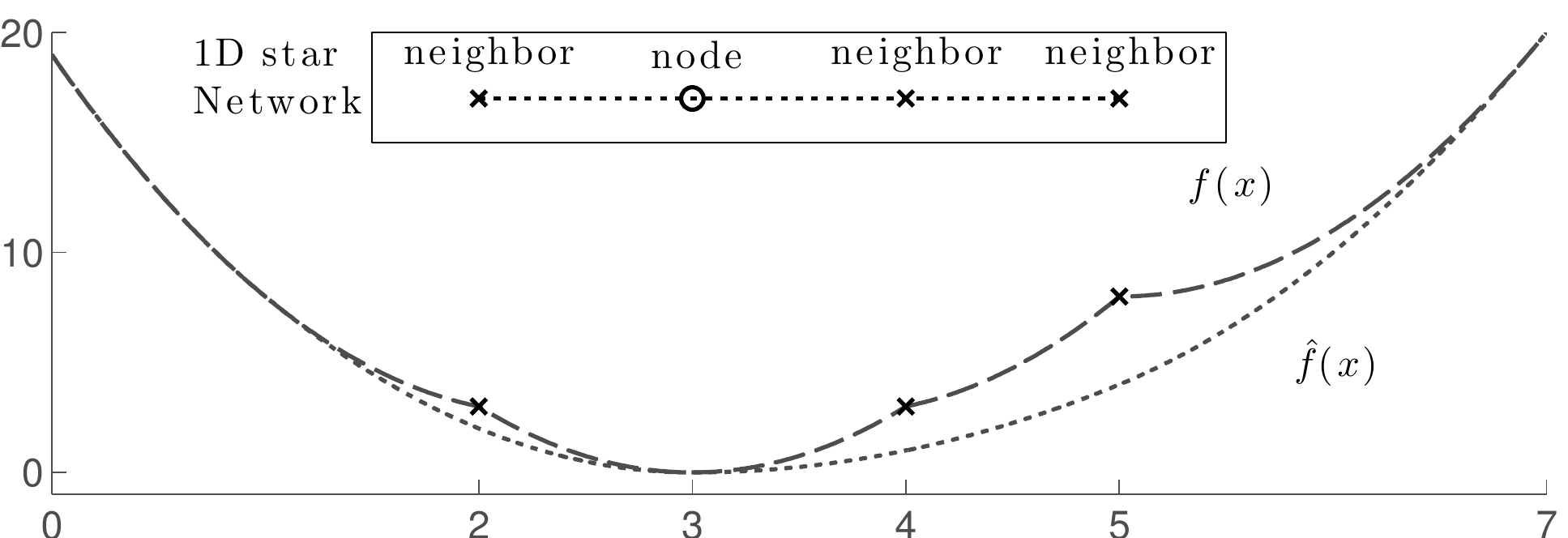}
  \caption{One-dimensional example of the quality of the approximation
    of the true nonconvex cost $f(x)$ by the convexified function
    $\hat{f}(x)$ in a star network. Here the node positioned at $x=3$ has 3
    neighbors.
    }
  \label{fig:nonconvexities}
\end{figure}
the bounds in~\eqref{eq:bound-tight}, and~\eqref{eq:udellBoyd}
\begin{table}[tb]
  \caption{Bounds on the optimality gap for the example in Figure~\ref{fig:nonconvexities}}
  \label{tab:bounds}
  \centering
  \begin{tabular}[h]{@{}ccc@{}}
    \toprule
    \textbf{$f^{\star} - \hat f^{\star}$} &
    \textbf{Equation~\eqref{eq:bound-tight}} &
    \textbf{Equation~\eqref{eq:udellBoyd}}\\\midrule
    0.0367&0.0487&3.0871\\
    \bottomrule
  \end{tabular}
\end{table}
averaged over 500 Monte Carlo trials are presented in
Table~\ref{tab:bounds}. The true average gap $f^{\star} - \hat
f^{\star}$ is also shown. In the Monte Carlo trials we sampled a zero
mean Gaussian random variable with $\sigma = 0.25$ and obtained a
noisy range measurement as described later by~\eqref{eq:noise}.  These
results show the tightness of the convexified function and how loose
the bound~\eqref{eq:udellBoyd} is when applied to our problem.

\subsection{Parallel method: convergence guarantees and iteration complexity}
\label{sec:conv-synchronous-method}

As Problem~\eqref{eq:cvx-prob-2} is convex and the cost function has a
Lipschitz continuous gradient, Algorithm~\ref{alg:synDR} is known to converge at the optimal rate
$O\left( k^{-2} \right)$
\cite{Nesterov1983},\cite{Nesterov2004}: $\hat f( x(k) ) - \hat f^\star \leq \frac{2 L_{\hat f}}{(k+1)^2} \left\| x(0) - x^\star \right\|^2.$

\subsection{Asynchronous method: convergence guarantees and iteration complexity}
\label{sec:conv-asynchronous-method}

To state the convergence properties of Algorithm~\ref{alg:asynDR} we
only need Assumption~\ref{th:connected-assumption}.
{ \begin{assumption}
  \label{th:connected-assumption}
  There is at least one anchor linked to some sensor and the graph~$\mathcal{G}$ is connected (there is a path between any two sensors).
\end{assumption}}
This assumption holds generally as one needs~$p+1$ anchors to
eliminate translation, rotation, and flip ambiguities while performing
localization in~$\reals^{p}$, {which exceeds the assumption
requirement}. We present two convergence results, ---
Theorem~\ref{th:asy-convergence}, and
Theorem~\ref{th:asy-convergence-GD} --- and the iteration complexity
analysis for Algorithm~\ref{alg:asynDR} in
Proposition~\ref{prop:iteration-cplx}. Proofs of the Theorems are
detailed in Appendix~\ref{sec:theorems}.

The following Theorem establishes the almost sure (\textit{a.s.})
convergence of Algorithm~\ref{alg:asynDR}.
\begin{theorem}[Almost sure convergence of Algorithm~\ref{alg:asynDR}]
  \label{th:asy-convergence}
  Let $\{x(k)\}_{k \in \naturals}$ be the sequence of points produced
  by Algorithm~\ref{alg:asynDR}, or by Algorithm~\ref{alg:asynDR} with
  the update~\eqref{eq:grad-step-update}, and let $\mathcal{X}^{\star}
  = \{x^{\star} : \hat f(x^{\star}) = \hat f^{\star}\}$ be the set of
  minimizers of function $\hat f$ defined
  in~\eqref{eq:cvx-prob-1}. Then it holds:
 \begin{equation}
    \lim_{k\to \infty} \mathrm{d}_{\mathcal{X}^{\star}}\left(x(k)\right) = 0,
    \qquad a.s.
\label{eq:convergence-th}
\end{equation}
\end{theorem}
In words, with probability one, the iterates $x(k)$ will approach the
set $\mathcal{X}^{\star}$ of minimizers of $\hat f$; this does not imply
that $\{x(k)\}_{k \in \naturals}$ will converge to one single $x^{\star}
\in \mathcal{X}^{\star}$, but it does imply that $\lim_{k \to \infty} \hat
f(x(k)) = \hat f ^{\star}$, since $\mathcal{X}^{\star}$ is a compact set, as
proven in Appendix~\ref{sec:auxil-lemmas},
Lemma~\ref{lem:basic-properties}.

\begin{theorem}[Almost sure convergence to a point]
  \label{th:asy-convergence-GD}
  Let~$\{x(k)\}_{k \in \naturals}$ be a sequence of points generated
  by Algorithm~\ref{alg:asynDR}, with the
  update~\eqref{eq:grad-step-update} in Step~\ref{alg:asynmin}, and
  let all nodes start computations with uniform probability. Then, with probability one,
  there exists a minimizer of~$\hat f$, denoted by $x^{\star} \in
  \mathcal{X}^{\star}$, such that
  \begin{equation}
    \label{eq:asy-convergence-GD}
    x(k) \rightarrow x^{\star}. 
  \end{equation}
\end{theorem}
This result tells us that the iterates of Algorithm~\ref{alg:asynDR}
with the modified Step~\ref{alg:asynmin} stated in
Equation~\eqref{eq:grad-step-update} not only converge to the solution
set, but also guarantees that they will not be jumping around the
solution set~$\mathcal{X}^{\star}$ (unlikely to occur in
Algorithm~\ref{alg:asynDR}, but not ruled out by the analysis). One of
the practical benefits of Theorem~\ref{th:asy-convergence-GD} is that
the stopping criterion can safely probe the stability of the estimates
along iterations.  To the best of our knowledge, this kind of strong
type of convergence (the whole sequence converges to a point in
${\mathcal X}^\star$) was not established previously in the context of
randomized approaches for convex functions with Lipschitz continuous
gradients, though it was derived previously for randomized
proximal-based minimizations of a large number of convex functions,
cf.~\cite[Proposition~9]{Bertsekas2011}. {We emphasize
  that what prevents the latter to apply to the exact version of
  Algorithm~\ref{alg:asynDR} is the ambiguity in choosing estimates
  when~$\mathcal{X}^{\star}$ is not a singleton. A possible approach
  to circumvent non-uniqueness of minimizers in~\eqref{eq:sl-problem}
  is to add a proximal term (as this makes the function strictly
  convex). However, the proximal terms tend to slow down
  convergence. Although overall strong convergence is still an open
  issue with this device, we saw in preliminary experiments that the
  proximal terms slowed down the speed of convergence (up to one order
  of magnitude of degradation in the iteration count).}

\begin{proposition}[Iteration complexity for Algorithm~\ref{alg:asynDR}]
  \label{prop:iteration-cplx}
  Let~$\{x(k)\}_{k \in \naturals}$ be a sequence of points generated
  by Algorithm~\ref{alg:asynDR}, with the
  update~\eqref{eq:grad-step-update} in Step~\ref{alg:asynmin}, and
  let the nodes be activated with equal probability. Choose $0 < \epsilon < \hat f( x(0) ) - \hat f^\star$ and $\rho \in (0,1)$.
There exists a constant $b(\rho, x(0))$ such that
  \begin{equation}
    \label{eq:iteration-cplx-p}
    \prob \left( \hat f(x(k)) - \hat f^{\star} \leq \epsilon \right)
    \geq 1-\rho
  \end{equation} for all 
  \begin{equation}
    \label{eq:iteration-cplx-k}
k \geq    K = \frac{2 n b(\rho, x(0))}{\epsilon} + 2 - n.
  \end{equation}
\end{proposition}
  The constant~$b( x(0), {\rho})$ can be computed from inequality~(19)
  in~\cite{LuXiao2013}; it depends only on the initialization and the chosen~$\rho$.
Proposition~\ref{prop:iteration-cplx} is saying that, with high
probability, the function value~$\hat f(x(k))$ for all~$k \geq K$ will
be at a distance~$\epsilon$ of the optimal, and the number of
iterations~$K$ depends inversely on the chosen~$\epsilon$.
\begin{proof}[Proof of Proposition~\ref{prop:iteration-cplx}]
  As~$\hat f$ is differentiable and has Lipschitz gradient, the result
  is trivially deduced from~\cite[Th.~2]{LuXiao2013}.
\end{proof}




\section{Numerical experiments}
\label{sec:experimental-results}

In this Section we present experimental results that demonstrate the
superior performance of our methods when compared with four state of
the art algorithms: Euclidean Distance Matrix (EDM) completion
presented in~\cite{OguzGomesXavierOliveira2011}, Semidefinite Program
(SDP) relaxation and Edge-based Semidefinite Program (ESDP)
relaxation, both implemented in~\cite{SimonettoLeus2014}, and a
sequential projection method (PM)
in~\cite{GholamiTetruashviliStromCensor2013} optimizing the same
convex underestimator as the present work, with a different
algorithm. The fist two methods --- EDM completion and SDP relaxation
--- are centralized, whereas the ESDP relaxation and PM are
distributed.

\subsubsection*{Methods}
\label{sec:methods}
We conducted simulations with two uniquely localizable geometric
networks with sensors randomly distributed in a two-dimensional square
of size~$1 \times 1$ with~$4$ anchors in the corners of the
square. Network~1 has~10~sensor nodes with an average node
degree\footnote{To characterize the used networks we resort to the
  concepts of \emph{node degree}~$k_{i}$, which is the number of edges
  connected to node~$i$, and \emph{average node degree}~$\langle k
  \rangle = 1/n \sum_{i=1}^{n}k_{i}$.} of~$4.3$, while network~2
has~50~sensor nodes and average node degree of~$6.1$.  The ESDP method
was only evaluated in network~1 due to simulation time constraints,
since it involves solving an SDP at each node, and each iteration. The
noisy range measurements are generated according to
\begin{align}
  \label{eq:noise}
  d_{ij} &=& | \|x_{i}^{\star} - x_{j}^{\star}\| + \nu_{ij} |, \; r_{ik} &=& | \|x_{i}^{\star} - a_{k}\| + \nu_{ik} |,
\end{align}
where~$x_{i}^{\star}$ is the true position of node~$i$, and~$\{\nu_{ij} :
i \sim j \in \mathcal{E}\} \cup \{\nu_{ik} : i \in \mathcal{V}, k \in
\mathcal{A}_{i}\}$ are independent Gaussian random variables with zero
mean and standard deviation~$\sigma$.  The accuracy of the algorithms
is measured by the original nonconvex cost value
in~\eqref{eq:snlOptimizationProblem} and by the Root Mean Squared
Error (RMSE) per sensor, defined as
{\begin{equation}
  \label{eq:rmse}
  \mathrm{RMSE} = \sqrt{\frac{1}{n}\left(\frac{1}{M} \sum_{m=1}^{M}
      \|x^{\star} - \hat x(m) \|^{2}\right)},
\end{equation}}
where~$M$ is the number of Monte Carlo trials performed. 

\subsection{{Assessment of the convex underestimator performance}}
\label{sec:relaxation-quality}

The first experiment aimed at exploring {the performance
  of the convex underestimator in~\eqref{eq:cvx-prob-1} when compared
  with two other state of the art convexifications.} For the proposed
disk relaxation~\eqref{eq:cvx-prob-1}, Algorithm~\ref{alg:synDR} was
stopped when the gradient norm~$\|\nabla \hat f(x)\|$ reached
$10^{-6}$ while both EDM completion and SDP relaxation were solved
with the default \emph{SeDuMi} solver~\cite{Sedumi98guide}
\texttt{eps} value of~$10^{-9}$, so that algorithm properties did not
mask the real quality of the relaxations.
\begin{figure}[h]
  \centering
  \includegraphics[width=\columnwidth]{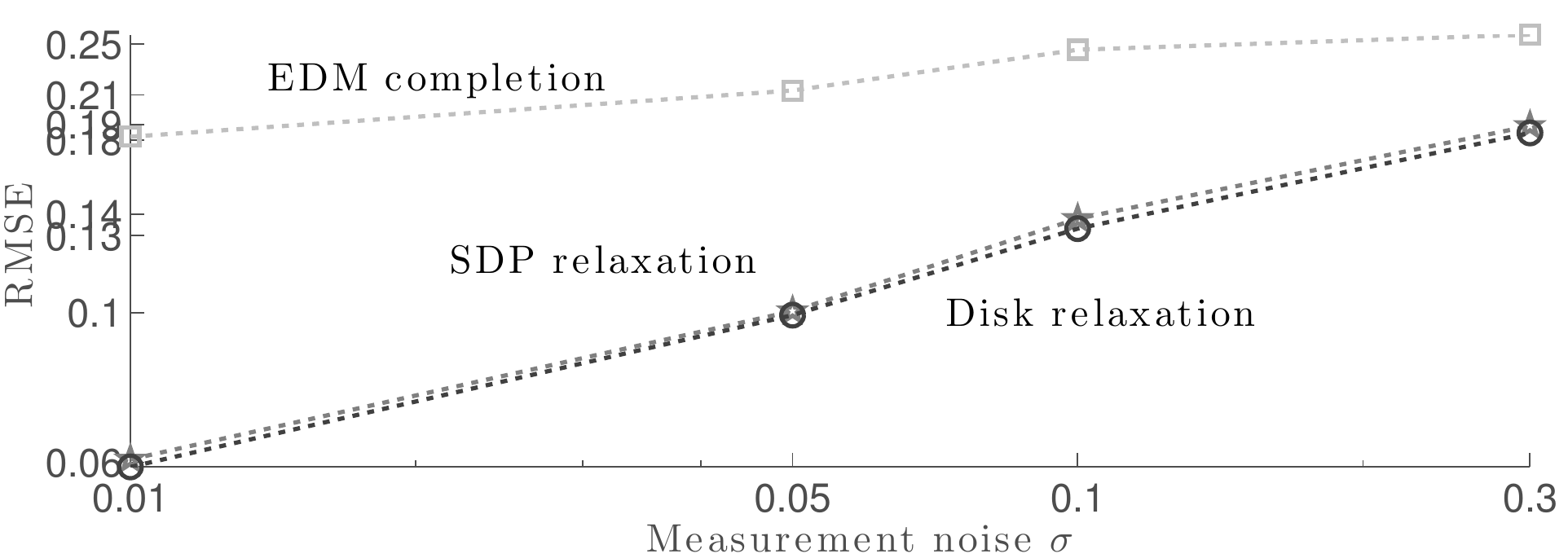}
  \caption{Relaxation quality: Root mean square error comparison of
    EDM completion in~\cite{OguzGomesXavierOliveira2011}, SDP
    relaxation in~\cite{SimonettoLeus2014} and the disk
    relaxation~\eqref{eq:cvx-prob-1}, used in the present work;
    measurements were perturbed with noise with different values for
    the standard deviation~$\sigma$. The disk relaxation approach
    in~\eqref{eq:cvx-prob-1} improved on the RMSE values of both EDM
    completion and SDP relaxation for all noise levels, even though it
    does not rely on the SDP machinery. The performance gap to EDM
    completion is substantial.}
  \label{fig:costFnQualityRMSEvsSigma}
\end{figure}
Figures~\ref{fig:costFnQualityRMSEvsSigma}
and~\ref{fig:costFnQualityRMSEvsTimeStdev0_1} report the results of
the experiment with~$50$ Monte Carlo trials over network~2 and
measurement noise with~$\sigma = [0.01, \: 0.05, \: 0.1, \: 0.3]$; so,
we had a total of~$200$ runs, equally divided by the~$4$ noise
levels. In Figure~\ref{fig:costFnQualityRMSEvsSigma} we can see that
the disk relaxation in~\eqref{eq:cvx-prob-1} has better performance
for all noise levels.
\begin{figure}[h]
  \centering
  \includegraphics[width=\columnwidth]{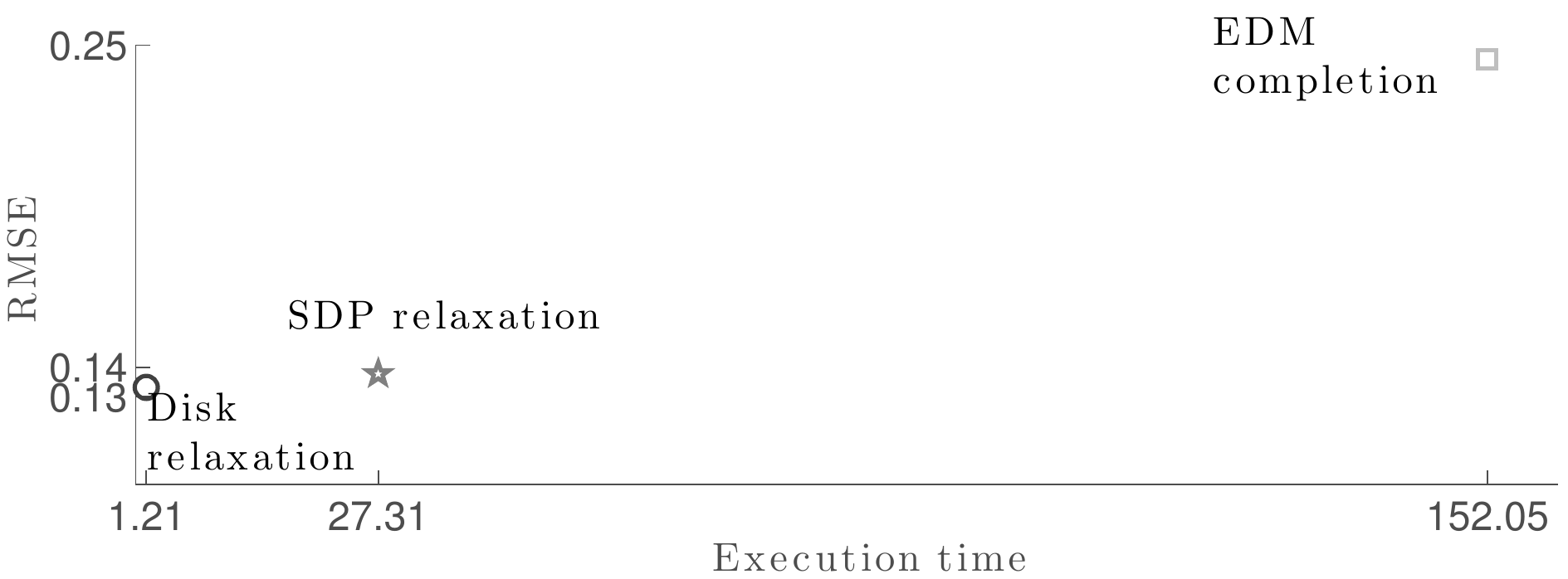}
  \caption{Relaxation quality: Comparison of the best achievable root
    mean square error \textit{versus} overall execution time of the
    algorithms. Measurements were contaminated with noise with $\sigma
    = 0.1$. Although disk relaxation~\eqref{eq:cvx-prob-1} has a
    distributed implementation, running it sequentially can be faster
    by one order of magnitude than the centralized methods.}
  \label{fig:costFnQualityRMSEvsTimeStdev0_1}
\end{figure}
Figure~\ref{fig:costFnQualityRMSEvsTimeStdev0_1} locates the results
of optimizing the three convex functions for the same problems in RMSE
\emph{versus} execution time, indicating the complexity of the
optimization of the considered costs. The convex
surrogate~\eqref{eq:cvx-prob-1} used in the present work combined with
our methods is faster by at least one order of magnitude.

\subsection{Performance of distributed optimization algorithms}
\label{sec:optim-algor-perf}

To measure the performance of the presented Algorithm~\ref{alg:synDR}
in a distributed setting we compared it with the state of the art
methods in~\cite{GholamiTetruashviliStromCensor2013} and the
distributed algorithm in~\cite{SimonettoLeus2014}. The results are
shown, respectively, in Figures~\ref{fig:GvsDRrmsevscommunications50s}
and~\ref{fig:DSvsDRrmsevsnoise}. The experimental setups were
different, since the authors proposed different stopping criteria for
their algorithms and, in order to do a fair comparison, we ran our
algorithm with the specific criterion set by each benchmark
method. Also, to compare with the distributed ESDP method
in~\cite{SimonettoLeus2014}, we had to use a smaller network
of~$10$~sensors because of simulation time constraints --- as the ESDP
method entails solving an SDP problem at each node, the simulation
time becomes prohibitively large, at least using a general purpose solver. The
number of Monte Carlo trials was~$32$, with~$3$~noise levels, leading
to~$96$ realizations for each noisy measurement.
\begin{figure}[h]
  \centering
  \includegraphics[width=\columnwidth]{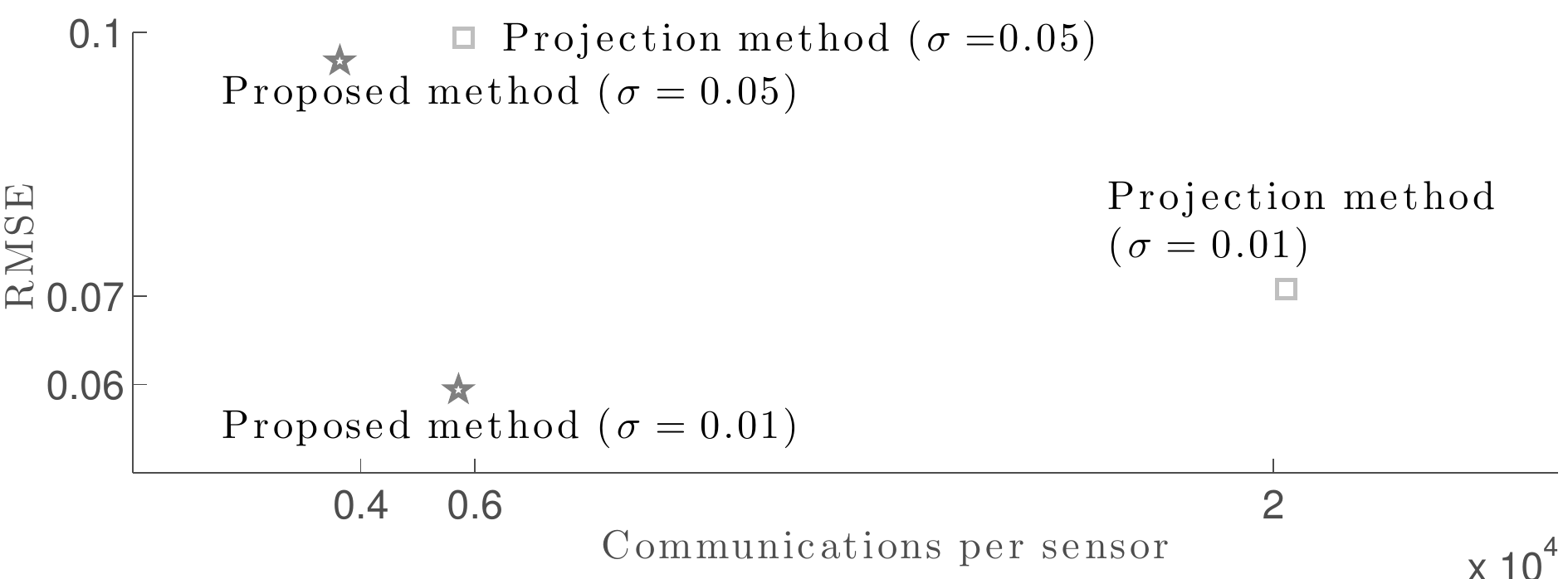}
  \caption{Performance of the proposed method {in Algorithm~1}
    and of the Projection method presented
    in~\cite{GholamiTetruashviliStromCensor2013}. The stopping
    criterion for both algorithms was a relative improvement
    of~$10^{-6}$ in the estimate. The proposed method uses fewer
    communications to achieve better RMSE for the tested noise
    levels. Our method outperforms the projection method with one
    forth of the number of communications for a noise level
    of~$0.01$.}
\label{fig:GvsDRrmsevscommunications50s}
\end{figure}
So, in the experiment illustrated in
Figure~\ref{fig:GvsDRrmsevscommunications50s}, the stopping criterion
for both the projection method and the presented method was the relative
improvement of the solution; we stress that this is not a distributed
stopping criterion,  we adopted it just for algorithm comparison. We can
see that the proposed method fares better not only in RMSE but,
foremost, in communication cost. The experiment comprised~$120$~Monte
Carlo trials and two noise levels.

\begin{figure}[h]
  \centering
  \includegraphics[width=\columnwidth]{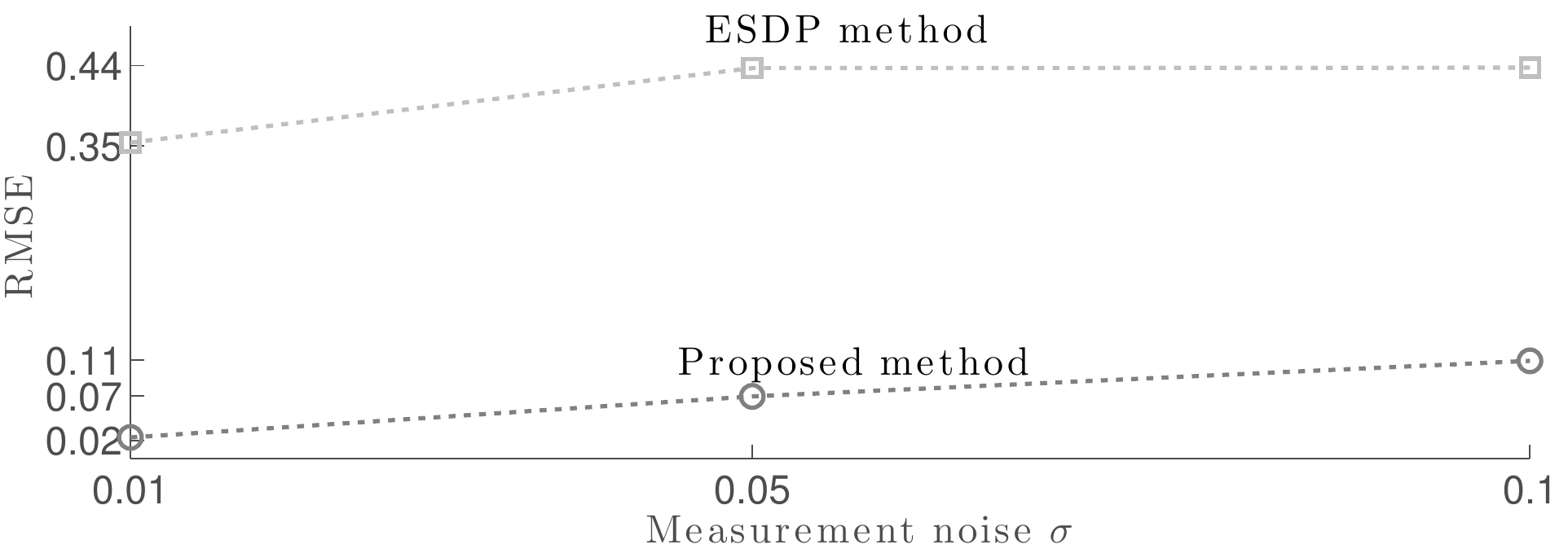}
  \caption{Performance of the proposed method {in
      Algorithm~1}
    and of the ESDP method in~\cite{SimonettoLeus2014}. The stopping
    criterion for both algorithms was the number of algorithm
    iterations. The performance advantage of the proposed method in
    Algorithm~\ref{alg:synDR} is even more remarkable when considering
    the number of communications presented in
    Table~\ref{tab:communications}.}
\label{fig:DSvsDRrmsevsnoise}
\end{figure}
\begin{table}[h]
  \caption{Number of communications per sensor for the results in
    Fig.~\ref{fig:DSvsDRrmsevsnoise}}
  \label{tab:communications}
  \centering
  \begin{tabular}[h]{@{}cc@{}}
    \toprule
    \textbf{ESDP method}&\textbf{{Algorithm~\ref{alg:synDR}}}\\\midrule
    21600&2000\\
\bottomrule
  \end{tabular}
\end{table}
From the analysis of both Figure~\ref{fig:DSvsDRrmsevsnoise} and
Table~\ref{tab:communications} we can see that the ESDP method is one
order of magnitude worse in RMSE performance, using one order of
magnitude more communications, than Algorithm~\ref{alg:synDR}.

{
\subsection{Performance of the asynchronous algorithm}
\label{sec:parall-asynchr-perf}

A second experiment consisted on testing the performance of the
parallel and the asynchronous flavors of our method, presented
respectively in Algorithms~\ref{alg:synDR} and~\ref{alg:asynDR},
{the latter with the exact update}. The metric was the
value of the convex cost function~$\hat f$ in~\eqref{eq:cvx-prob-1}
evaluated at each algorithm's estimate of the minimum. To have a fair
comparison, both algorithms were allowed to run until they reached a
preset number of communications.
\begin{figure}[h]
  \centering
  \includegraphics[width=\columnwidth]{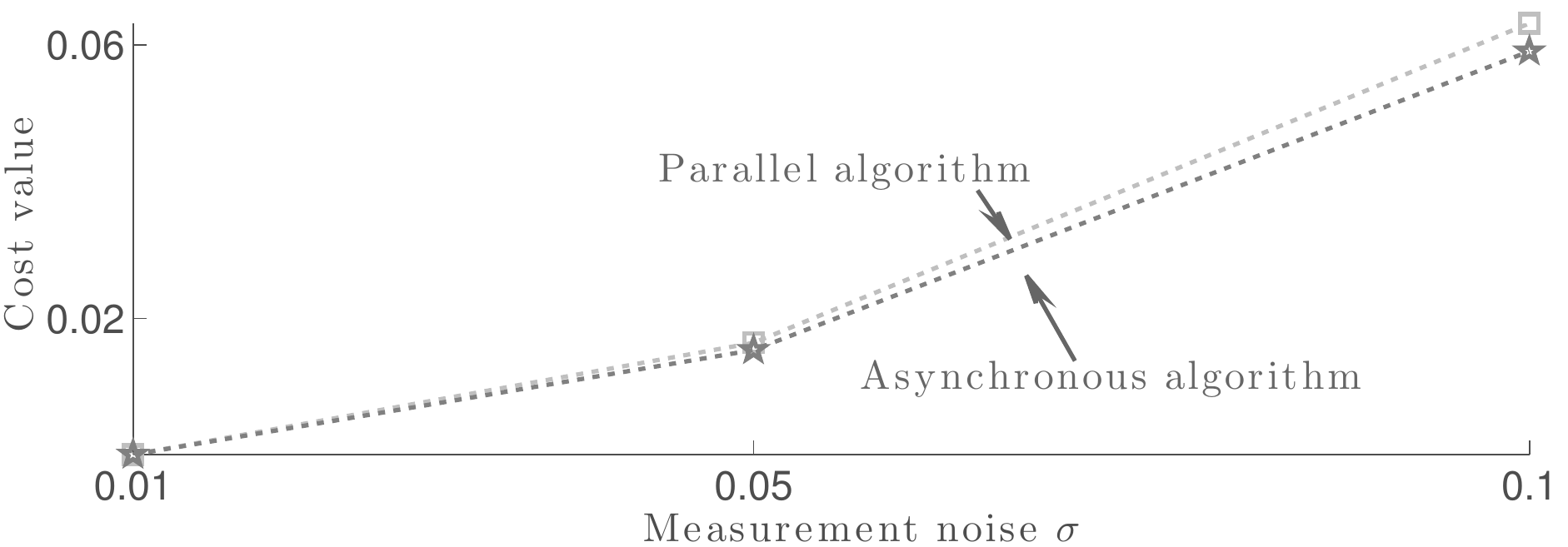}
  \caption{Final cost of the parallel Algorithm~\ref{alg:synDR} and
    its asynchronous counterpart in Algorithm~\ref{alg:asynDR}
    {with an exact update} for the same number of
    communications. Results for the asynchronous version degrade less
    than those of the parallel one as the noise level increases. The
    stochastic Gauss-Seidel iterations prove to be more robust to
    intense noise.}
  \label{fig:parallelvsAsynCvxCost}
\end{figure}
In Figure~\ref{fig:parallelvsAsynCvxCost} we present the effectiveness
of both algorithms in optimizing the disk relaxation cost
in~\eqref{eq:cvx-prob-1}, with the same amount of communications. We
chose the uniform probability law for the random variables~$\xi_{k}$
representing the sequence of updating nodes in the asynchronous
version of our method. Again, we ran $50$ Monte Carlo trials, each
with $3$ noise levels, thus leading to $150$ samplings of the noise
variables in~\eqref{eq:noise}. 

}

\section{Concluding remarks}
\label{sec:concluding-remarks}

Experiments in Section~\ref{sec:experimental-results} show that our
method is superior to the state of the art in all measured
indicators. While the comparison with the projection method published
in~\cite{GholamiTetruashviliStromCensor2013} is favorable to our
proposal, it should be further considered that the projection method
has a different nature when compared to ours: it is sequential, and
such algorithms will always have a larger computation time than
parallel ones, since nodes run in sequence; moreover, this
computation time grows with the number of sensors while parallel
methods retain similar speed, no matter how many sensors the
network has.

When comparing with a distributed and parallel method similar to
Algorithm~\ref{alg:synDR}, like the ESDP method
in~\cite{SimonettoLeus2014} we can see one order of magnitude
improvement in RMSE for one order of magnitude fewer communications of
our method --- and this score is achieved with a simpler,
easy-to-implement algorithm, performing simple computations at each
node that are well suited to the kind of hardware commonly found in
sensor networks.

There are some important questions not addressed here. For example, it
is not clear what influence the number of anchors and their spatial
distribution can have in the performance of the proposed and state of
the art algorithms. Also, an exhaustive study on the impact of varying
topologies and number of sensors could lead to interesting
results. {Some preliminary experiments show that all
  convex relaxations experience some performance degradation when
  tested for robustness to sensors outside the convex hull of the
  anchors. This issue has been noted by several authors, but a more
  exhaustive study exceeds the scope of this paper.}

But with the data presented here one can already grasp the advantages of
our fast and easily implementable distributed method, where the
optimality gap of the solution can also be easily quantified, and which
offers two implementation flavours for different localization needs.

\section*{Acknowledgements}
\label{sec:acknowledgements}
The authors would like to thank Pinar Oguz-Ekim and Andrea Simonetto
for providing the Matlab implementation of the methods in their
papers. We also thank the reviewers for the very interesting questions
raised, and productive recommendations.

\appendices

\section{Convex envelope} \label{sec-cvxenv}

{
We show that the function in~\eqref{eq:ball-sq-dist} is the convex envelope of the function in~\eqref{eq:sphere-sq-dist}.
 Refer to $\alpha$ as the function in~\eqref{eq:sphere-sq-dist} and $\beta$ as the function in~\eqref{eq:ball-sq-dist}. 
We show that $\alpha^{\star \star} = \beta$ where $f^\star$ denotes the Fenchel conjugate of a function~$f$, cf.~\cite[Cor. 1.3.6, p. 45, v. 2]{UrrutyMarechal1993}.

 We start by computing $\alpha^\star$: \begin{eqnarray*}
\alpha^\star(s) & = & \sup_{z}\, s^\top z - \alpha(z) \\ & = & \sup_z\, s^\top z - \left( \frac{1}{2} \inf_{ \left\| y \right\| = d_{ij}} \left\| z - y \right\|^2 \right)
\\ & = & \sup_z \sup_{\left\| y \right\| = d_{ij}}\, s^\top z - \frac{1}{2} \left\| z - y \right\|^2 \\ & = & \sup_{\left\| y \right\| = d_{ij}} \sup_z\, 
s^\top z - \frac{1}{2} \left\| z - y \right\|^2 \\ & = & \sup_{\left\| y \right\| = d_{ij}}\,\frac{1}{2} \left\| s \right\|^2 + s^\top y \\ & = & \frac{1}{2} \left\| s \right\|^2 + d_{ij} \left\| s \right\|.
\end{eqnarray*}
Thus, $\alpha^\star$ is the sum of two closed convex functions: $\alpha^\star = g + h$ where $g(s) = \frac{1}{2} \left\| s \right\|^2$ and $h(s) = d_{ij}\left\| s 
\right\|$. Note that $h(s) = \sigma_{\text{B}(0,d_{ij})}(s)$ where $\sigma_C(s) = \sup\{ s^\top x\,:\, x \in C \}$ denotes the support function of a set $C$. Thus, using \cite[Th. 2.3.1, p. 61, v. 2]{UrrutyMarechal1993}, we have 
$$\alpha^{\star \star}(z) = \inf_{z_1 + z_2 = z}\,g^\star(z_1) + h^\star(z_2).$$ 
Since $g^\star(z_1) = \frac{1}{2} \left\| z_1 \right\|^2$ \cite[Ex. 1.1.3, p. 38, v. 2]{UrrutyMarechal1993}
 and $h^\star(z_2) = i_{\text{B}_{ij}}(z_2)$ \cite[Ex. 1.1.5, p. 39, v. 2]{UrrutyMarechal1993} where $i_C(x) = 0$ if $x \in C$ and $i_C(x) = +\infty$ if $x \not\in C$ denotes the indicator of a set~$C$, we conclude that \begin{eqnarray*}
\alpha^{\star \star}(z) & = & \inf_{z_1 + z_2 = z}\, \frac{1}{2} \left\| z_1 \right\|^2 + i_{\text{B}_{ij}}(z_2) \\ & = & \inf_{z_2 \in \text{B}_{ij}}\,\frac{1}{2} \left\| z - z_2 \right\|^2 \\ & = & \beta(z).
\end{eqnarray*}
}

\section{Lipschitz constant of $\nabla \phi_{\text{B}_ij}$}
\label{sec-lips}

{ 
We prove the inequality in~\eqref{eq:sq-dist-lipschitz-grad}:  \begin{equation} \left\| \nabla \phi_{\text{B}_{ij}} (x) - \nabla \phi_{\text{B}_{ij}} (y) \right\| \leq \left\| x - y \right\| \label{kf} \end{equation} 
where $\nabla \phi_{\text{B}_{ij}}(z) = z - \text{P}_{\text{B}_{ij}}(z)$ and $\text{P}_{\text{B}_{ij}}(z)$ is the projector onto  $\text{B}_{ij} = 
\left\{ z \in {\mathbb R}^p\,:\,\left\| z \right\| \leq d_{ij} \right\}$.  Squaring both sides of~\eqref{kf} gives the equivalent inequality 
\begin{equation}
2  ( \text{P}(x) - \text{P}(y) )^\top ( x  - y )- \left\| \text{P}(x) - \text{P}(y) \right\|^2 \geq 0
\label{kf2}
\end{equation}
where, to simplify notation, we let $\text{P}(z) := \text{P}_{\text{B}_{ij}}(z)$.
Inequality~\eqref{kf2} can be rewritten as
\begin{IEEEeqnarray}{rCl}
( \text{P}(x) - \text{P}(y) )^\top ( x - y ) + \left( \text{P}(x) - \text{P}(y) \right)^\top ( \text{P}(y) - y ) \nonumber \\ + ( \text{P}(x) - \text{P}(y) )^\top ( x - \text{P}(x) ) \geq 0.
\label{kf3}
\end{IEEEeqnarray}
By the properties of projectors onto closed convex sets, $( z - \text{P}(z) )^\top ( w - \text{P}(z) ) \leq 0$, for any $w \in \text{B}_{ij}$ and any $z$, cf.~\cite[Th. 3.1.1, p. 117, v. 1]{UrrutyMarechal1993}. Thus, the last two terms on the left-hand side of~\eqref{kf3} are nonnegative. Moreover, the first term is nonnegative due to~\cite[Prop. 3.1.3, p. 118, v. 1]{UrrutyMarechal1993}. Inequality~\eqref{kf3} is proved.}

\section{Auxiliary Lemmas}
\label{sec:auxil-lemmas}

In this Section we establish basic properties of
Problem~\eqref{eq:cvx-prob-2} in Lemma~\ref{lem:basic-properties} and
also two technical Lemmas, instrumental to prove our convergence
results in Theorem~\ref{th:asy-convergence}.

\begin{lemma}[Basic properties]
  \label{lem:basic-properties}
  Let $\hat f$ as defined in~\eqref{eq:cvx-prob-1}. Then the following
  properties hold.
  \begin{enumerate}
  \item $\hat f$ is coercive;
  \item $\hat f^{\star} \geq 0$ and $\mathcal{X}^{\star} \neq
    \varnothing$;
  \item $\mathcal{X}^{\star}$ is compact;
  \end{enumerate}
\end{lemma}
\begin{proof}
  \begin{enumerate}[leftmargin=*,noitemsep,topsep=0pt,parsep=0pt,partopsep=0pt]
  \item By Assumption~\ref{th:connected-assumption} there is a path
    from each node~$i$ to some node~$j$ which is connected to an
    anchor~$k$. If $\|x_{i}\| \to \infty$ then there are two cases:
    (1) there is at least one edge $t \sim u$ along the path from~$i$
    to~$j$ where~$\|x_{t}\| \to \infty$ and~$\|x_{u}\|\not \to
    \infty$, and so $\mathrm{d}^{2}_{\mathrm{B}_{tu}}(x_{t}-x_{u}) \to
    \infty$; (2) if $\|x_{u}\| \to \infty$ for all~$u$ in the path
    between~$i$ and~$j$, in particular we have~$\|x_{j}\| \to \infty$
    and so~$\mathrm{d}^{2}_{\mathrm{B_{a}}_{jk}}(x_{j}) \to \infty$,
    and in both cases~$\hat f \to \infty$, thus,~$\hat f$ is coercive.
    
  \item Function~$\hat f$ defined in~\eqref{eq:cvx-prob-1} is a sum of
    squares, it is continuous, convex and a real valued function,
    lower bounded by zero; so, the infimum~$\hat f^{\star}$ exists and
    is non-negative. To prove this infimum is attained
    and~$\mathcal{X}^{\star} \neq \varnothing$, we consider the set~$T
    = \{x : \hat f(x) \leq \alpha \}$; $T$ is a sublevel set of a
    continuous, coercive function and, thus, it is compact.  As~$\hat
    f$ is continuous, by the Weierstrass Theorem, the value~$p =
    \inf_{x \in T} \hat f(x)$ is attained; the equality~$\hat
    f^{\star} = p$ is evident.

  \item $\mathcal{X}^{\star}$ is a sublevel set of a continuous coercive
    function and, thus, compact. \qedhere
  \end{enumerate}
\end{proof}

\begin{lemma}
  \label{lem:vanish-grad}
  Let~$\{x(k)\}_{k \in \naturals}$ be the sequence of iterates of
  Algorithm~\ref{alg:asynDR}, or of Algorithm~\ref{alg:asynDR} with
  the update~\eqref{eq:grad-step-update}, and~$\nabla \hat f\left(
    x(k) \right)$ be the gradient of function~$\hat f$ evaluated at
  each iterate. Then,
  \begin{enumerate}
  \item $ \displaystyle
      \sum_{k \geq 1} \| \nabla \hat f\left( x(k) \right)\|^{2} <
      \infty, \; a.s.;$
    \item $\displaystyle \nabla \hat f\left( x(k) \right) \to 0, \;
      a.s.$ 
  \end{enumerate}
\end{lemma}

\begin{proof}
  Let~$\mathcal{F}_{k} = \sigma \left(x(0), \cdots, x(k) \right)$ be
  the sigma-algebra generated by all the algorithm iterations until
  time~$k$. We are interested in~$\expect\left[ \hat
    f\left(x(k)\right)| \mathcal{F}_{k-1}\right]$, the expected value
  of the cost value of the~$k$th iteration, given the knowledge of the
  past~$k-1$ iterations.  Firstly, let us examine function~$\phi :
  \reals^{p} \to \reals$, the slice of~$\hat f$ along a coordinate
  direction,~$\phi(y) = \hat f( x_1, \ldots, x_{i-1}, y, x_{i+1}, \ldots, x_n)$.
 As~$\hat f$ has
  Lipschitz continuous gradient with constant~$L_{\hat f}$, so will~$\phi$:
$
    \|\nabla \phi(y) - \nabla \phi(z) \| \leq L_{\hat f} \|y-z\|,
 $
  for all~$y$ and~$z$, and, thus, it will inherit the
  property  \begin{equation}
    \label{eq:quadratic-majorizer}
    \phi(y) \leq \phi(z) + \left< \nabla \phi (z), y-z \right> + \frac{L_{\hat f}}{2} \|y-z\|^{2}. 
  \end{equation}
{  Inequality~\eqref{eq:quadratic-majorizer} is known as the Descent Lemma~\cite[Prop. A.24]{bertsekas1999nonlinear}.}
  The minimizer of the quadratic upper-bound
  in~\eqref{eq:quadratic-majorizer} is
    $z - \frac{1}{L_{\hat f}} \nabla\phi(z)$,
  which can be plugged back in~\eqref{eq:quadratic-majorizer},
  obtaining
  \begin{equation}
    \label{eq:bound-to-optimal-slice}
    \phi^{\star} \leq \phi\left(z-\frac1{L_{\hat f}} \nabla\phi(z)\right) \leq
    \phi(z) - \frac 1{2 L_{\hat f}}\|\nabla\phi(z)\|^{2}.
  \end{equation}
  In the sequel, for a given $x = ( x_1, \ldots, x_n )$, we let $$\hat
  f_i^\star(x_{-i}) = \inf\{ \hat f(x_1, \ldots, x_{i-1}, z, x_{i+1},
  \ldots, x_n )\,:\, z \}.$$ Going back to the
  expectation~$\expect\left[ \hat
    f\left(x(k)\right)|\mathcal{F}_{k-1}\right] = \sum_{i=1}^{n}P_{i}
  \hat f_{i}^{\star}\left(x_{-i}(k-1)\right)$, we can bound it from
  above, recurring to~\eqref{eq:bound-to-optimal-slice}, by
  \begin{align}
    \nonumber &\sum_{i=1}^{n} P_{i}\left( \hat
      f(x(k-1)) -
      \frac1{2 L_{\hat f}}\|\nabla_{i}\hat{f}(x(k-1))\|^{2} \right) \\
    \nonumber & = \hat f(x(k-1)) - \frac1{2 L_{\hat f}} \sum_{i=1}^{n}
    P_{i}\|\nabla_{i}\hat{f}(x(k-1))\|^{2}\\
    \label{eq:expectation-k-cost} 
    &\stackrel{(a)}{\leq} \hat f(x(k-1)) - \frac
    {P_{\mathrm{min}}}{2 L_{\hat f}}\|\nabla\hat f(x(k-1))\|^{2},
  \end{align}
  where we used~$0 < P_{\mathrm{min}} \leq P_{i}$, for all~$i \in
  \{1,\cdots,n\}$ in~$(a)$. To alleviate notation, let~$g(k) = \nabla
  \hat f(x(k))$; we then have
  \begin{equation*}
    \label{eq:extended-g}
    \|g(k)\|^{2} = \sum_{i \leq k} \|g(i)\|^{2} - \sum_{i \leq k-1}\|g(i)\|^{2},
  \end{equation*}
  and adding~$\frac {P_{\mathrm{min}}}{2L}\sum_{i \leq
    k-1}\|g(i)\|^{2}$ to both sides of the inequality
  in~\eqref{eq:expectation-k-cost}, we find that
  \begin{equation}
    \label{eq:supermartingale}
    \expect\left[ Y_{k} | \mathcal{F}_{k-1}\right] \leq Y_{k-1},
  \end{equation}
  where $Y_{k} = \hat f(x(k)) + \frac{P_{\mathrm{min}}}{2L}
  \sum_{i\leq k-1} \|g(i)\|^{2}$.
  Inequality~\eqref{eq:supermartingale} defines the
  sequence~$\left\{Y_{k}\right\}_{k \in \naturals}$ as a
  supermartingale. As~$\hat f(x)$ is always non-negative, then~$Y_{k}$
  is also non-negative and so~\cite[Corollary 27.1]{JacodProtter01},
  \begin{equation*}
    \label{eq:convergence-a-s}
    Y_{k} \to Y, \; a.s.
  \end{equation*}
  In words, the sequence~$Y_{k}$ converges almost surely to an
  integrable random variable~$Y$. This entails that $ \sum_{k \geq 1}
  \|g(k)\|^{2} < \infty, \; a.s., $ and so, $ g(k) \to 0, \;
  a.s.\qedhere $ { The previous arguments show that
    Lemma~\ref{lem:vanish-grad} holds for
    Algorithm~\ref{alg:asynDR}. To show that
    Lemma~\ref{lem:vanish-grad} also holds for
    Algorithm~\ref{alg:asynDR} with the
    update~\eqref{eq:grad-step-update} it suffices to redefine $$\hat
    f_i^\star (x_{-i}) := \hat f\left( x_1, \ldots, x_i -
      \frac{1}{L_{\hat f}} \nabla_i \hat f( x ), \ldots, x_n
    \right).$$ As the second inequality
    in~\eqref{eq:bound-to-optimal-slice} shows, we have the bound
$$ \hat f_i^\star (x_{-i}(k-1)) \leq \hat f( x(k-1) ) -  \frac{1}{L_{\hat f}} \left\| \nabla_i \hat f\left( x(k-1) \right) \right\|^2$$
and the rest of the proof holds intact.
 }
\end{proof}

\begin{lemma}
  \label{lem:cost-value-convergence}
  Let $\{ x(k) \}_{k \in \naturals}$ be one of the sequences generated with probability one according to Lemma~\ref{lem:vanish-grad}.
  Then,
  \begin{enumerate}
  \item The function value decreases to the optimum: $\hat f (x(k))
    \downarrow \hat f^{\star};$
  \item There exists a subsequence of~$\{x(k)\}_{k \in \naturals}$ converging
    to a point in~$\mathcal{X}^{\star}$: $x(k_{l}) \to y, \; y \in
    \mathcal{X}^{\star}$.
  \end{enumerate}
\end{lemma}
\begin{proof}
  As~$\hat f$ is coercive, then the sublevel set~$\mathcal{X}_{\hat f}
  = \left\{x : \hat f(x) \leq \hat f(x(0)) \right\}$ is compact and,
  because~$\hat f (x(k))$ is non increasing, all elements
  of~$\{x(k)\}_{k \in \naturals}$ belong to this set. From the
  compactness of~$\mathcal{X}_{\hat f}$ we have that there is a
  convergent subsequence~$x(k_{l}) \to y$. We evaluate the gradient at
  this accumulation point,~$\nabla \hat f(y) = \lim_{l \to \infty}
  \nabla \hat f(x(k_{l}))$, which, by assumption, vanishes, and we
  therefore conclude that~$y$ belongs to the solution
  set~$\mathcal{X}^{\star}$. Moreover, the function value at this
  point is, by definition, the optimal value.
\end{proof}

\section{Proofs of Theorems in Section~\ref{sec:convergence-analysis}}
\label{sec:theorems}
Equipped with the previous lemmas, we are now ready to prove the
Theorems stated in Section~\ref{sec:convergence-analysis}.

\begin{proof}[Proof of Theorem~\ref{th:asy-convergence}]
  Suppose the distance does not converge to zero. Then, there exists
  an~$\epsilon > 0$ and some subsequence~$\{x(k_{l})\}_{l \in \naturals}$ such
  that~$\mathrm{d}_{\mathcal{X}^{\star}}(x(k_{l})) > \epsilon$. But,
  as~$\hat f$ is coercive (by Lemma~\ref{lem:basic-properties}),
  continuous, and convex, and whose gradient, by
  Lemma~\ref{lem:vanish-grad}, vanishes, then by
  Lemma~\ref{lem:cost-value-convergence}, there is a subsequence
  of~$\{x(k_{l})\}_{l \in \naturals}$ converging to a point
  in~$\mathcal{X}^{\star}$, which is a contradiction.
\end{proof}

\begin{proof}[Proof of Theorem~\ref{th:asy-convergence-GD}] Fix an
arbitrary point $x^\star \in {\mathcal X}^\star$. We start by proving
that the sequence of squared distances to~$x^{\star}$ of the estimate
produced by Algorithm~\ref{alg:asynDR}, with the update defined in
Equation~\eqref{eq:grad-step-update}, converges almost surely; that
is, the sequence $\{ \left\| x(k) - x^\star \right\|^2 \}_{k \in
\naturals}$ is convergent with probability one. We have
  \begin{eqnarray}
    \label{eq:expect-dist-gd} \lefteqn{\expect\left[\|x(k) -
x^{\star}\|^{2} | \mathcal{F}_{k-1} \right] =} \\ & & \sum_{i = 1}^n
\frac{1}{n} \left\| x(k-1) - \frac{1}{L_{\hat f}} g_i (k-1) - x^\star
\right\|^2 \nonumber
\end{eqnarray} where $g_i(k-1) = ( 0, \ldots, 0, \nabla_i \hat f(
x(k-1) ), 0, \ldots , 0)$ and ~$\mathcal{F}_{k} = \sigma\left(x(1),
  \ldots, x(k)\right)$ is the sigma-algebra generated by all iterates
until time~$k$. Expanding the right-hand side
of~\eqref{eq:expect-dist-gd} yields
  \begin{eqnarray*} \lefteqn{\left\| x(k-1) - x^\star \right\|^2 +
\frac{1}{ n L_{\hat f}^2 } \left\| \nabla \hat f ( x(k-1) )
\right\|^2} \\ & & - \frac{2}{n L_{\hat f}} ( x(k-1) - x^\star )^\top
\nabla \hat f( x(k-1) ).
  \end{eqnarray*} Since $( x(k-1) - x^\star )^\top \nabla \hat f(
x(k-1) ) = ( x(k-1) - x^\star )^\top \left( \nabla \hat f( x(k-1) ) -
\nabla f ( x^\star ) \right) \geq 0$, we conclude
that \begin{eqnarray*} \lefteqn{\expect\left[\|x(k) - x^{\star}\|^{2}
| \mathcal{F}_{k-1} \right]} \\ & \leq & \left\| x(k-1) - x^\star
\right\|^2 + \frac{1}{ n L_{\hat f}^2 } \left\| \nabla \hat f ( x(k-1)
) \right\|^2.
  \end{eqnarray*} Now, as proved in Lemma~\ref{lem:vanish-grad}, the
sum~$\sum_{k} \|\nabla \hat f(x(k))\|^{2}$ converges almost
surely. Thus, invoking the result in~\cite{RobbinsSiegmund1985}, we
get that~$\|x(k) - x^{\star}\|^{2}$ converges almost surely.

{ We can now invoke the technique at the end of the proof
  of~\cite[Prop.~9]{Bertsekas2011} to conclude that $x(k)$
  converges to some optimal point $x^\star$.}

\end{proof}

\bibliographystyle{IEEEtran} \bibliography{IEEEabrv,biblos}

\begin{IEEEbiography}[{\includegraphics[width=1in,height
=1.25in,clip,keepaspectratio]{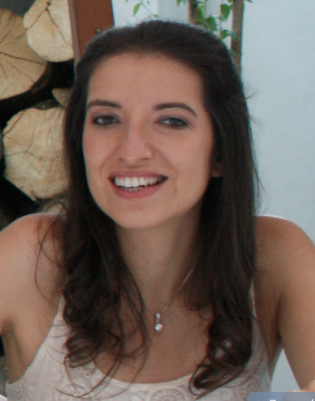}}]{Cl\'{a}udia Soares}
(S'10) received the M.S.\ degree in electrical and computer engineering from Instituto
Superior Tecnico (IST), Lisbon, Portugal in 2007. 
She is currently working toward the Ph.D. degree in
electrical and computer engineering
in the Signal and
Image Processing Group of the Institute for Systems and Robotics, IST.
Her research interests include distributed optimization and sensor networks.
\end{IEEEbiography}


\begin{IEEEbiography}[{\includegraphics[width=1in,height
=1.25in,clip,keepaspectratio]{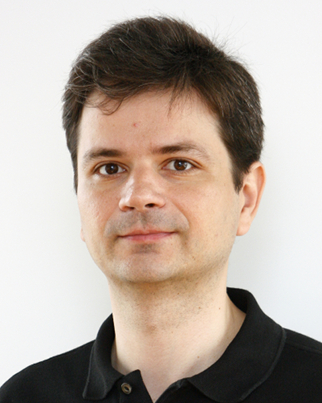}}]{Jo\~{a}o Xavier}
(S'97 - M'03) received the Ph.D. degree in Electrical and
Computer Engineering from Instituto Superior Tecnico (IST), Lisbon,
Portugal, in 2002. Currently, he is an Assistant Professor in the
Department of Electrical and Computer Engineering, IST. He is also a
Researcher at the Institute of Systems and Robotics (ISR), Lisbon,
Portugal.  His current research interests are in the area of
optimization and statistical inference for distributed systems.
\end{IEEEbiography}


\begin{IEEEbiography}[{\includegraphics[width=1in,height
=1.25in,clip,keepaspectratio]{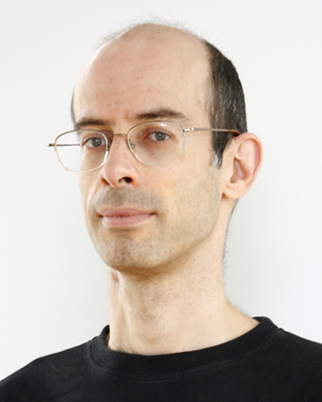}}]{Jo\~{a}o Gomes}
(S'95-M'03) received the Diploma, M.S. and
Ph.D. degrees in electrical and computer engineering from Instituto
Superior Técnico (IST), Lisbon, Portugal, in 1993, 1996 and 2002,
respectively. He joined the Department of Electrical and Computer
Engineering of IST in 1995, where he is presently an Assistant
Professor. Since 1994 he has also been a researcher in the Signal and
Image Processing Group of the Institute for Systems and Robotics, in
Lisbon. He currently serves as an Associate Editor for signal
processing and communications in the IEEE Journal of Oceanic
Engineering. His research interests include channel identification and
equalization in wireless communications, underwater communications and
acoustics, fast algorithms for adaptive filtering, and sensor
networks.
\end{IEEEbiography}


\end{document}